\documentclass{article}
\usepackage[utf8]{inputenx}
\usepackage{tikz-cd}
\usepackage{indentfirst}
\usepackage{microtype}
\usepackage{amsfonts}
\usepackage{mathtools}
\usepackage[english]{babel}
\usepackage{euscript}
\usepackage{setspace}
\usepackage{amsthm}
\usepackage[bottom]{footmisc}
\usepackage{bbm}
\usepackage{xpatch}
\usepackage{pgfplots}
\pgfplotsset{compat = 1.15}
\usetikzlibrary{arrows.meta}
\usepackage{amsmath,url}
\usepackage{wrapfig}
\usepackage{dsfont} 
\usepackage{caption}
\usepackage{float}
\usepackage{bm}
\usepackage [a4paper, top=3.5cm, bottom=3.5cm, left=3.3cm, right=3.3cm] {geometry}
\usepackage{comment}
\usepackage{todonotes}
\usepackage{amssymb}
\usepackage{overpic}
\usepackage{mathrsfs} 
\usepackage{hyperref}
\usepackage{makeidx}

\usepackage{resmes}

\usepackage{graphicx}
\usepackage[export]{adjustbox}
\graphicspath{ {./images/} }

\usepackage{lettrine}
\theoremstyle{definition}
\newtheorem{definition}{Definition}[section]
\newtheorem{theorem}[definition]{Theorem}

\newtheorem{example}[definition]{Example} 
\newtheorem{rk}{Remark}[section]
\newtheorem{lemma}[definition]{Lemma}

\newtheorem{corollary}[definition]{Corollary}

\newtheorem{hyp}[definition]{Hypothesis}

\let\oldexample\example
\renewenvironment{example}
  {\oldexample\pushQED{\qed}\ignorespaces} 
  {\hfill$\blacklozenge$\par}                  

\usepackage{stackengine}

\newcommand{\jump}{\hfill\break}
\newcommand{\E}{\mathbb{E}}
\newcommand{\R}{\mathbb{R}}
\newcommand{\N}{\mathbb{N}}

\newcommand{\T}{\mathbb{T}}

\newcommand{\nN}{\textnormal{N}}
\newcommand{\e}{\textnormal{e}}

\renewcommand{\d}{\textnormal{d}}
\newcommand{\D}{\textnormal{D}}

\newcommand{\tb}{\vert\!\vert\!\vert}

\newcommand*\samethanks[1][\value{footnote}]{\footnotemark[#1]}
\title{Synchronization for the Rough Kuramoto Model}
\author{Alexandra Blessing Neam\c tu\thanks{University of Konstanz, Department of Mathematics and Statistics,  Universit\"atsstra\ss{}e~10 78464 Konstanz, Germany. E-Mail: alexandra.blessing@uni-konstanz.de, dennis.rudik@uni-konstanz.de }, Christian Kuehn\thanks{Technical University of Munich, School of Computation, Information and Technology, Department of Mathematics, Boltzmannstraße 3, 85748 Garching, Germany. E-Mail: ckuehn@ma.tum.de, giacomo.landi@tum.de}, Giacomo Landi\samethanks, Dennis Rudik\samethanks[1]}
\date{April 2, 2026}

\begin{document}

\maketitle
\begin{abstract}
We study the local synchronization of phases and frequencies for the Kuramoto model driven by rough noise. In particular, we prove exponential convergence towards synchronization and we give the explicit rate of convergence and quantify the size of the random basin of attraction. Furthermore, we show that the long time behavior of the system is determined by the evolution of phases' mean. Our result relies on the use of a Lyapunov function, capable of overriding the particular structure of the noise, taking in account only its intensity. Finally, we illustrate our analytical results and possible extensions with the help of numerical simulations.
\end{abstract}
\begin{section}{Introduction}
Synchronization is one of the most striking collective phenomena in nonlinear science: a population of interacting particles, despite starting from spread initial phases, may spontaneously evolve toward a coherent synchronized regime. This behavior appears across a wide range of disciplines, from physics \cite{wiesenfeld1996synchronization} and biology \cite{buck1988synchronous} to engineering \cite{rohden2012self} and neuroscience \cite{gray1994synchronous}, and has motivated extensive mathematical investigation  over the past decades. A paradigmatic model capturing this transition is the Kuramoto model, introduced in the seminal work of Kuramoto \cite{kuramoto2003chemical}. Since its introduction, the model has become a cornerstone of synchronization theory, with a vast literature devoted to its analytical and numerical study \cite{strogatz2000kuramoto,acebron2005kuramoto}.
In this paper we consider the Kuramoto model perturbed by a \emph{rough noise}, i.e., 
\begin{equation}\label{RKM}
    \d\theta_i(t)=\left[\varpi_i+\frac{K}{N}\sum_{j=1}^Na_{ij}\sin{(\theta_j(t)-\theta_i(t))}\right]\d t+(G(\theta(t))\d \textbf{W}_t)_i ,\qquad\forall i=1,...,N\tag{RKM}
\end{equation}
where $\theta_i\in\mathbb{S}^1$ are the phases of the oscillators, $K\in\R^+$ is the intensity of the coupling, $\varpi_i\in\R$ are the natural frequencies\footnote{Note that usually in the literature natural frequencies are denoted with $\omega_i$. Here we prefer to use $\varpi$ since $\omega$ is already used for the stochastic realizations.} that characterize each oscillator, the $a_{ij}$ are the entries of the adjacency matrix $\mathcal{A}$, $G$ is the noise function governing our rough noise $\textbf{W}$ (to be specified later), and $N$ denotes the number of oscillators.\jump\jump
Stability analysis for the Kuramoto model has always been an important topic. For the deterministic case there is a lot of literature regarding stability. One of the first papers using Lyapunov functions in order to state the stability of the synchronous state is \cite{van1993lyapunov}. Furthermore,  \cite{ha2010complete} establish the complete synchronization of the oscillators when they are contained in half a circle. This result is extended and modified for initial data on the entire circle in \cite{dong2013synchronization}. The same result has been proven in \cite{benedetto2014complete} using other techniques. Moreover, the number of possible stable solutions of the Kuramoto model has been investigated for identical oscillators in \cite{taylor2012there,ling2019landscape,taylor2012there} and for non-identical oscillators in \cite{arenas2023number}. For general positive symmetric weights synchronization is studied in \cite{jadbabaie2004stability}, while for the Kuramoto model on oriented and signed graphs we refer to \cite{delabays2019kuramoto,burylko2014bifurcation,yoon2026stability,hong2011kuramoto}. Instead of finitely many particles, another common direction is to study the stability of invariant sets in the Kuramoto model via the mean-field limit, see e.g.,~\cite{chiba2015proof,benedetto2014complete,benedetto2016exponential,dietert2018mathematics}. However, in realistic applications there are finitely many oscillators. In addition, in this context the oscillators are usually not only subject to smooth deterministic coupling forces but also to irregular perturbations, which may arise from highly oscillatory environments or intrinsic stochastic fluctuations, leading to the stochastic Kuramoto model~\cite{acebron2005kuramoto,strogatz1991stability}. These considerations led to a surge recent activity in trying to generalize deterministic synchronization results to the case of stochastic perturbations focusing on white noise perturbations; see e.g.~\cite{wu2020global,liu2022stochastic}.
To make models even more realistic and flexible, one can broaden the class of noise forcing processes. Examples are systems perturbed by a rough noise, such as fractional Brownian motion and other non-semimartingale processes. These noise processes can be treated in a unified way using rough path theory. Since its development, rough path theory has become a fundamental tool in modern stochastic analysis \cite{friz2014course,lyons1998differential}, with growing influence in nonlinear dynamics and interacting particle systems.\jump\jump
Our contribution is to extend Kuramoto model synchronization results to the case where the perturbation is given by rough noise. To the best of the authors’ knowledge, this is the first work that rigorously proves the robustness of synchronization in the Kuramoto model under rough perturbations. There are formal arguments regarding white noise~\cite{wu2020global,liu2022stochastic}, which one may upon modification make rigorous. Here we will focus on the more general case of rough noises, also improving significantly upon previous results and mathematical details. In particular, we prove the synchronization of identical oscillators on connected graphs using a suitable Lyapunov function. This is natural in the context of rough differential equations, see \cite{duc2025strong}, where a stronger notion of a Lyapunov function than in~\cite{khasminskii2012stochastic} was used in order to prove local exponential stability of equilibria for rough systems.~For further results on local exponential stability for such systems we refer to~\cite{stabilityyoung,robert}. 
In our situation, whenever the graph is not connected we are able to prove the synchronization on every connected subgraph, even if from our numerical simulations it seems plausible that one can obtain a unique synchronous state if the particles are still interacting via the noise. For non-identical oscillators we are only able to prove the synchronization of the frequencies, which is consistent with the existing literature~\cite{chopra2005synchronization,chopra2009exponential,jadbabaie2004stability}. We also show  the remarkable result  that if there are enough algebraic constrains on the noise, then the rough dynamics is confined into a deterministic hyperplane.\jump\jump
The paper is organized as follows. In Section \ref{Sec_main} we will state all the main results of the paper. In particular, in Subsection \ref{subsec_syn}, we introduce the main hypotheses under which we establish local synchronization of the phases for \eqref{RKM}. To achieve this, we present in Subsection \ref{Doss-Suss} the Doss–Sussmann transformation, a standard tool for proving existence of solutions to rough differential equations (RDEs). Finally, in Subsection \ref{subsec_synproof}, we prove the local synchronization for the \eqref{RKM}.
In Subsection \ref{subsec_split} we prove for balanced signed graphs that the system synchronizes into two different clusters. In Subsection \ref{syn_fre} we suggest an RDE for the frequencies $\varpi_i$ and there we prove the synchronization of the oscillators' frequencies. Furthermore, in Subsection \ref{subsec_det} we show how, up to assuming a stronger symmetry hypothesis on $G$, one obtains that the whole rough dynamics is contained in a deterministic plane. In the last Subsection \ref{Syn_rot} we study the synchronization of particles' phases, after changing set of hypotheses, towards a random variable that does not admit a limit in the long time regime. In particular there we really show that the long time behavior of the system is determined by the mean of the noise functions. We finish with Section \ref{Num_sec}, where we cross-validate our theoretical results and we discuss further developments that could be obtained with a refined theoretical set of tools.
\end{section}
\begin{section}{Main Results}\label{Sec_main}
Before proceeding with the main part of the paper we would like to recall notation related to the adjacency matrix. The adjacency matrix encodes all the interactions among the $N$ particles, representing the information of the underlying graph $\mathcal{G}=(V,E)$ where $V=\{v_1,v_2,\ldots,v_N\}$ is the set of vertices representing the $N$ particles, $E=\{(v_i,v_j)\vert v_i,v_j\in V\}$ is the edge set of the graph $\mathcal{G}$ and $\mathcal{N}=\{1,\ldots,N\}$. We have that $a_{ij}>0$ if there is a link from $v_i$ to to $v_j$ and the particles are "agreeing", $a_{ij}<0$ if there is a link but the particles do not agree and $a_{ij}=0$ if they are not interacting.
\begin{subsection}{Setting for Local Synchronization of the Phases}\label{subsec_syn}
In this Subsection, we state the assumptions and derive a reformulation of \eqref{RKM} that will be used to establish local synchronization of the phases in Subsection \ref{subsec_synproof}. For our purpose we start by assuming a particular set of hypotheses for our model. Later on we will show how we can remove certain restrictions and adapt our results accordingly. We have decided to start in this way because under this set of hypotheses we will still see the essential mechanism of the proof, that will lead to the anticipated generalization in a straightforward way. Therefore, for the moment, we assume the following hypotheses:
\begin{hyp}\hfill
\begin{description}
\item[(\textbf{A})\label{A}] We suppose that $\mathcal{A}$ is symmetric with entries $a_{ij}\geq0$ and that $\mathcal{A}$ represents a connected graph, so that for every $i,j=1,\ldots,N$ exists a sequence of indices $(i_l)_{1}^n$ with $i_1=i$ and $i_n=j$ such that $a_{i_l,i_{l+1}}>0$ for every $l=1,\ldots n-1$.

\item[(\textbf{B})\label{B}]  We suppose that all the natural frequencies are identical, i.e.~$\varpi_i=\underline{\varpi}\in\R$. 

\item[(\textbf{H}\textsubscript{\textbf{W}})\label{W}] For a given $\gamma\in(\frac{1}{3},\frac{1}{2}]$, $\mathcal{W}_t(\omega)\in\R^m$ is a stochastic process with stationary increments, of which almost all realizations $W$ belong to the space $C^\gamma(\R,\R^m)$ of $\gamma$-Hölder continuous paths, such that $\mathcal{W}$ can be lifted into a rough path $\mathbf{W}=(W,\mathbb{W})$ of a stochastic process $(\mathcal{W}_{\cdot}(\omega),\mathbb{W}_{\cdot,\cdot}(\omega))$ with stationary increments, and the estimate
        \begin{equation}\label{kolmo}
            \E\left(\|\mathcal{W}_{s,t}\|^p+\|\mathbb{W}_{s,t}\|^q\right)\leq C_{T,\gamma}|t-s|^{p\gamma},\qquad\forall s,t\in[0,T]
        \end{equation}
        holds for any $[0,T]$, for a fixed $p$ such that $p\gamma\geq1,q=\frac{p}{2}$ and some constant $C_{T,\gamma}$. Moreover we suppose that $\Xi$, the Wiener-shift defined on the probability space $\Omega$ is ergodic. Both $\Xi$ and $\Omega$ are precisely defined in Appendix \ref{App_A}. We refer to Appendix \ref{App_A} for the precise definitions of the norms and related objects.

\item[(\textbf{H}\textsubscript{\textbf{G}})\label{G}]  $G$ is in $C^3_b(\T^N,\mathcal{L}(\R^m,\mathbb{T}^N))$ where $\mathcal{L}(\R^m,\mathbb{T}^N)$ denotes the space of linear operators from $\R^m$ into the $N$-dimensional torus $\mathbb{T}^N$. We also assume that $G(0)=0$ and that $G$ is rotational invariant, i.e.,
        \begin{equation}\label{rot_inv}
            G(\theta+a)=G(\theta)\qquad\forall\theta\in\T^N,\forall a\in\R.
        \end{equation}
        In the following we will denote by $G_{ij}(\theta)$ the components of $G$ for $i=1,\ldots,N$ and $j=1,\ldots,m$. Moreover, we define $C_G$ as
        $$C_G\coloneqq \max\bigg\{\;\|G\|_\infty,\;\|\D G\|_\infty,\;\|\D^2G\|_\infty,\;\|\D^3G\|_\infty\bigg\}.$$
        Note that the derivatives of $G$ are defined as follows
        \begin{align*}
            &\D G:\mathbb{T}^N\longrightarrow\mathcal{L}(\T^N,\mathcal{L}(\R^m,\T^N))\simeq\mathcal{L}(\T^N\times\R^m,\T^N),\\
            &\D^2G:\mathbb{T}^N\longrightarrow\mathcal{L}(\T^N,\mathcal{L}(\T^N,\mathcal{L}(\R^m,\T^N)))\simeq\mathcal{L}(\T^N\times\T^N\times\R^m,\T^N),\\
            &\D^3G:\mathbb{T}^N\longrightarrow\mathcal{L}(\T^N,\mathcal{L}(\T^N,\mathcal{L}(\T^N,\mathcal{L}(\R^m,\T^N))))\simeq\mathcal{L}(\T^N\times\T^N\times\T^N\times\R^m,\T^N).
        \end{align*}
\end{description}
\end{hyp}
\begin{rk}\label{rk1}
    An example of stochastic \textcolor{black}{process} that satisfy \nameref{W} is fractional Brownian motion with Hurst index $H>1/3$, see \cite[Theorem 10.4, Ex. 10.11]{friz2014course} for $H>1/3$ and \cite[Chapter 15]{friz2010multidimensional} for $H\in(1/4,1/3)$. For fractional Brownian motion, it is well-known that $\Xi$ is ergodic, see e.g. \cite[Theorem 20]{GS}. 
    Moreover, due to \nameref{G} we can perform a change of coordinates $\theta\longmapsto\theta+t\underline{\varpi}$ so that the system is reduced to
    $$\d\theta_i(t)=\left[\frac{K}{N}\sum_{j=1}^Na_{ij}\sin{(\theta_j(t)-\theta_i(t))}\right]\d t+(G(\theta(t))\d \textbf{W}_t)_i ,\qquad\forall i=1,...,N.$$ Furthermore, we have that the mean phase of the system $\Theta\coloneqq\frac{1}{N}\sum_{i=1}^N\theta_i$ is governed by
    \begin{align}\label{mean_phase}
    \begin{split}
        \d\Theta(t)=\frac{1}{N}\sum_{i=1}^N\d\theta_i(t)&=\frac{1}{N}\sum_{i=1}^N\frac{K}{N}\sum_{j=1}^Na_{ij}\sin{(\theta_j(t)-\theta_i(t))}\d t+\frac{1}{N}\sum_{i=1}^N(G(\theta(t))\d \textbf{W}_t)_i=\\
        &=0+\frac{1}{N}\sum_{i=1}^N\sum_{j=1}^m G_{ij}(\theta(t))\d\textbf{W}^j_t.
    \end{split}
    \end{align}
    where the first term vanishes due to the symmetry of $\mathcal{A}$ and the oddness of $\sin(\cdot)$.
    A natural class of noise functions that satisfy \nameref{G} is 
    \begin{equation}\label{noise}
        G_{ij}(\theta)=\sigma\sum_{k=1}^Nb_{ik}\sin(\theta_i-\theta_k)^{\tilde{n}} \qquad\forall i=1,\ldots,N,\;\forall j=1,\ldots,m\;;\sigma\in\R^+;\text{ for some $\tilde{n}\in\N^+$}
    \end{equation}
    where $\sigma$ is a scalar that control the intensity of the noise and $(\mathcal{B})_{ij}=b_{ij}$ are the entries of the adjacency matrix of the noise, with $\mathcal{B}$ symmetric once again.
    \end{rk}

\begin{rk}\label{example_matrix}
An important example of an adjacency matrix $\mathcal{A}$ satisfying \nameref{A} is the case of all-to-all coupling, so when $a_{ij}=1$ for all $i,j=1,\ldots,N$. Other important examples are the $k$-neighbor adjacency matrices and the classical undirected adjacency matrices where $a_{ij}=1$ if the particle $i$ and $j$ interact between each other and $a_{ij}=0$ otherwise.
\end{rk}
\begin{rk}
Since the drift term belongs to $C_b^\infty$, and $G\in C^3_b$, one can show that \eqref{RKM} admits global solutions defined for all times in the interval $[0,+\infty)$. See for instance, \cite{bailleul2020non,davie2008differential,RS:17} for non-explosion criteria ensuring global existence of solutions, or \cite{lyons1998differential} for more classical global existence results for rough differential equations.
\end{rk}
Now we perform a change of variables
\begin{equation}\label{change_coord}
    \hat{\theta}(t)\coloneqq\theta(t)-\Theta(t)=(\theta_1-\Theta,\cdots,\theta_N-\Theta)(t)
\end{equation}
so that our system reads
\begin{equation}\label{reduced RKM}
\begin{cases}
\displaystyle
    \d\hat{\theta}_i(t)=\left[\frac{K}{N}\sum_{j=1}^Na_{ij}\sin{(\hat{\theta}_j(t)-\hat{\theta}_i(t))}\right]\d t+\sum_{j=1}^m \Tilde{G}_{ij}(\hat{\theta}(t))\d \textbf{W}^j_t ,\qquad\forall i=1,...,N\\
    \displaystyle
    \sum_{i=1}^N\hat{\theta}_i=0
\end{cases}
\end{equation}
where we have defined
$$\Tilde{G}_{ij}(\theta)\coloneqq G_{ij}(\theta)-\frac{1}{N}\sum_{k=1}^N G_{kj}(\theta).$$
In particular now our system is reduced to an $N-1$ dimensional system and our modified noise function $\Tilde{G}$ satisfies
$$\sum_{i=1}^N\Tilde{G}_{ij}(\theta)=0\qquad\forall\; j=1,\ldots,m.$$
 In the following we will treat $\mathbb{S}^1$ as the closed interval $[-\pi,\pi]$. At first sight, this restriction might seem an artifact. Actually if some of the $\hat{\theta}_i$ hit the boundary of the interval and jump, e.g., from $-\pi$ to $\pi$, we have that the argument of our functions does not change. But no jump occurs in our setting, indeed we will study and prove the asymptotic stability of the equilibrium $\hat{\theta}^\star=0$ in a suitable neighborhood $I_\delta=[-\delta,\delta]$ with $\delta\in[0,\pi/2\textcolor{black}{)}$.
\end{subsection}
\begin{subsection}{The Doss-Sussmann transformation}\label{Doss-Suss}
In order to investigate the local exponential stability of \eqref{RKM} we construct a Lyapunov function. For alternative approaches for rough (partial) differential equations that use cut-off arguments to modify the coefficients and discrete Gronwall inequalities, we refer to \cite{stabilityyoung,robert}. Note that these approaches assume that $G(0)=\D G(0)=0$ (and similarly for the drift) which is not necessary for the Lyapunov-type argument.

 The strategy that we follow is based on a Doss-Sussmann transformation / flow transformation which is briefly described below. This transformation has been frequently used in order to transform an SDE with additive or linear multiplicative noise into an ODE with random non-autonomous coefficients.
 In our case, we apply it to a rough differential equation to remove its drift term, similar to~\cite[Theorem 4.3]{RS:17}
 ~In particular, our aim is to derive a Lyapunov function for a rough differential equation of the form
\begin{equation}\label{RDE}
    \d y_t=f(y_t)\d t+g(y_t)\d \textbf{W}_t,
\end{equation}
under the standard assumptions on the coefficients $f \in C^1_b$ and $g\in C^3_b$ ensuring its well-posedness. In order to do so one considers the solution $\phi_{\cdot,a}(\mathbf{W},\phi_a)$ of the \textit{pure} rough differential equation 
\begin{equation}\label{pure RDE}
    \d\phi_t=g(\phi_t)\d \textbf{W}_t\qquad t\in[a,b],\;\phi_a\in\R^d
\end{equation}
with initial value $\phi_a$. Such a solution is $C^1$ with respect to the initial datum, which results from \cite[Theorem 3]{CL}. The idea is to solve equation \eqref{RDE} on small intervals $[\tau_k,\tau_{k+1}]$ such that 
$$\lambda\coloneqq16C_pC_g\tb\textbf{W}\tb_{p-\text{var},[\tau_k,\tau_{k+1}]}\leq1$$
where $C_p$ is a constant coming from the sewing lemma \cite[Theorem 4.10]{friz2014course} and $C_g$ is defined analogously to $C_G$. Then one concatenates all the intervals to obtain the conclusion on the whole interval $[0,T]$. The Doss-Sussmann technique guarantees that there exists a path-dependent transformation $y_t=\phi_{t,\tau_k}(\mathbf{W},z_t)$  that is a bijection between a solution of \eqref{RDE} on a certain interval $[\tau_k,\tau_{k+1}]$ and a solution of the associated ordinary differential equation
\begin{equation}\label{Associate ODE}
\begin{cases}
\displaystyle
    \dot{z}_t=\left[\frac{\partial\phi}{\partial z}(t,\tau_k,\mathbf{W},z_t)\right]^{-1}f(\phi_{t,\tau_k}(\mathbf{W},z_t))=(\textnormal{Id}+\psi_t)f(z_t+\eta_t)\qquad t\in[\tau_k,\tau_{k+1}]\\
    \displaystyle z_{\tau_k}=y_{\tau_k}
\end{cases}
\end{equation}
where we set 
$$\eta_t\coloneqq y_t-z_t=\phi_{t,\tau_k}-z_t,$$
$$\phi_t\coloneqq\left[\frac{\partial\phi}{\partial z}(t,\tau_k,\mathbf{W},z_t)\right]^{-1}-\textnormal{Id}.$$
It can be shown that, for every $\lambda\in(0,1)$ the intervals $[\tau_k,\tau_{k+1}]$ can be chosen such that
$$\|\eta_t\|,\|\phi_t\|\leq\lambda,\qquad\forall t\in[\tau_k,\tau_{k+1}].$$
In particular the choice of these so-called greedy times  $\{\tau_k\}_{k\in\N}$ first introduced in \cite[Section 4]{CLL} is given by
$$\tau_0=0,\qquad\tau_{k+1}\coloneqq\inf\left\{t>\tau_k\;\bigg\vert\;\tb\textbf{W}\tb_{p-\text{var},[\tau_k,t]}=\frac{\lambda}{16C_pC_g}\right\}\wedge T,$$
\textcolor{black}{
where $\wedge$ denotes the minimum of the corresponding times. Throughout the paper, we reserve the more classical notation $\min$ for the minimum among elements.}
The function $\textcolor{black}{\overline{\nN}(\cdot)}$ below counts the number of the greedy times up to time $T$, i.e.,
$$\textcolor{black}{\overline{\nN}}\left(\frac{\lambda}{16C_pC_g},\textbf{W},[0,T]\right)\coloneqq\sup\{k\in\N\;\vert\;\tau_k\leq T\}.$$
Whenever we will use a different $\tilde{p}$-variation norm to define the greedy times, and consequently $\textcolor{black}{\overline{\nN}}(\cdot)$, we will denote the resulting function by $\textcolor{black}{\overline{\nN}}_{\tilde{p}}(\cdot)$. According to \cite[Lemma 3.2]{duc2024stability} for every $a,b\in\R,a\leq b$ and for every $\gamma>0$ the following inequalities hold
\begin{equation}\label{N_up}
    \textcolor{black}{\overline{\nN}}(\gamma,\textbf{W},[a,b])\leq1+\frac{1}{\gamma^p}\tb\textbf{W}\tb^p_{p-\text{var},[a,b]},
\end{equation}
\begin{equation}\label{N_down}
    \textcolor{black}{\overline{\nN}}(\gamma,\textbf{W},[a,b])\geq\frac{1}{\gamma}\tb\textbf{W}\tb_{p-\text{var},[a,b]}.
\end{equation}
We now state an auxiliary result which will be used later on.
\begin{lemma}\label{lemma1}
    Consider the rough Kuramoto model \eqref{RKM} satisfying the hypotheses \nameref{A}, \nameref{B}, \nameref{W}, \nameref{G},
     then the Doss-Sussmann transformation preserves the zero mean subspace, i.e., if we have a  solution $\hat{\theta}(t)$ of \eqref{reduced RKM} then its Doss-Sussmann associated solution satisfies
    $$\sum_{i=1}^N z_i(t)=0\qquad \text{for all } t \geq 0.$$
\end{lemma}
\begin{proof}
    In order to prove this Lemma we just simply notice that $\eta_t$ can be rewritten as
    \begin{equation}\label{int_res}
         \eta_t=\phi_{t,\tau_k}(\mathbf{W},z_t)-z_t=\int_{\tau_k}^t \tilde{G}(\phi_{s,\tau_k}(\mathbf{W},z_t))\d \textbf{W}_s,\qquad t\in[\tau_k,\tau_{k+1}]
    \end{equation}
    so that
    $$\sum_{i=1}^N z_i(t)=\sum_{i=1}^N y_i(t)-\sum_{i=1}^N\eta_i(t)=0-\int_{\tau_k}^t \sum_{i=1}^N    \sum_{j=1}^m\tilde{G}_{ij}(\phi_{s,\tau_k}(\mathbf{W},z_t))\d \textbf{W}^m_s=0,$$
    completing the proof. An alternative proof is proposed in Appendix~\ref{App_B}.
\end{proof}
\end{subsection}
\begin{subsection}{Local Synchronization of the Phases}\label{subsec_synproof}
We are now ready to prove the local synchronization for the phases of \eqref{RKM}. In order to do so we start by giving the definition of Lyapunov (exponential) stability for a rough differential equation.
\begin{definition}\label{defStab}\textit{\textbf{Stability}}: a solution $y^\star$ of \eqref{RDE} is called (Lyapunov) stable if for any $\varepsilon>0$ there exists a positive random variable $r=r(\omega)>0$ such that for any initial value $y_0\in\R^d$ satisfying $\|y_0-y^\star_0\|<r(\omega)$ the solution of \eqref{RDE} starting from $y_0$ exists on the whole half line $t\in[0,\infty)$ and the following inequality holds
$$\sup_{t\geq0}\|y_t-y^\star_t\|<\varepsilon.$$\\
\textit{\textbf{Exponential stability}}:  a solution $y^\star$ of \eqref{RDE} is called exponentially stable if it is stable and there exist two positive random variables $r(\omega)>0$ and $r^\star(\omega)>0$ and a positive constant $\mu>0$ such that for any initial value $y_0$ satisfying $\|y_0-y^\star_0\|<r(\omega)$ the solution of \eqref{RDE} starting from $y_0$ exists on the whole half line $t\in[0,\infty)$ and the following inequality
$$\|y_t-y_t^\star\|\leq r^\star(\omega)e^{-\mu t}$$
holds for all $t\geq0$.
\end{definition}
So our goal is to prove the following theorem.
\begin{theorem}\label{main_th}
    Under the hypotheses \nameref{A}, \nameref{B}, \nameref{W}, \nameref{G} then there exists a constant $\lambda_0$ such that for any $C_{\textcolor{black}{\Tilde{G}}}<\lambda_0$ we have that $\hat{\theta}^\star=0$ is almost surely exponentially \textcolor{black}{stable} in the sense of Definition \ref{defStab} on $I_\delta$, with $\delta\in[0,\pi/2)$. Moreover, always referring to Definition~\ref{defStab}, we have that the rate of convergence $\mu$ \textcolor{black}{can be chosen arbitrarily in the interval}
    \begin{equation}\label{rate_bound}
        0<\mu<\frac{K}{N}C_\delta\lambda_2-\textcolor{black}{L_f}(2+C_{\textcolor{black}{\Tilde{G}}})C_{\textcolor{black}{\Tilde{G}}}-C_{\textcolor{black}{\Tilde{G}}}\E \textcolor{black}{\overline{\nN}}\left(\frac{1}{16C_p},\textbf{W},[0,1]\right)
    \end{equation}
    with $\lambda_2$ the smallest positive eigenvalue of the Laplacian matrix of $\mathcal{A}$, \textcolor{black}{$C_\delta\in\left(0,1\right]$} is a constant depending only on $\delta$ \textcolor{black}{and $L_f$ is the Lipschitz constant of the deterministic drift of \eqref{RKM}}.
    Moreover $r(\omega)$ and $r^\star(\omega)$ have an explicit expression, which depends on the choice of $\mu$.
\end{theorem}
\begin{proof}
\textit{\textbf{Step 1}}:
        In order to prove the exponential stability we want to show that we have a Lyapunov function that satisfies the hypotheses of \cite[Theorem 23]{duc2025strong}. In our case it is sufficient to consider the following Lyapunov function
        $$V(\theta):=\|\theta\|=\sqrt{\sum_{i=1}^N\theta_i^2}.$$
        Let us fix $\delta\in(0,\pi/2)$ and consider the compact sub-domain $I_\delta=[-\delta,\delta]\subseteq[-\pi,\pi]$. Of course we have that $V\in C^1(I_\delta\setminus\{0\};\R)\cap C^0(I_\delta;\R)$ and it satisfies trivially some of technical hypotheses required~\cite[Theorem 23]{duc2025strong} given by: 
       \begin{equation}\label{Lya1}
           \|\nabla V(\theta)\|\leq L_V\qquad\forall\theta\in I_\delta
       \end{equation}
       \begin{equation}\label{Lya2}
        \alpha\|\theta\|+V(0)\leq V(\theta)\leq\beta(1+\|\theta\|)\qquad\forall\theta\in I_\delta
       \end{equation}
        with $L_V=\alpha=\beta=1$. Lastly we need to prove that $\exists\; d>0$ such that
        \begin{equation}\label{Lyapunov_cond}
            \langle\nabla V(\theta),f(\theta)\rangle\leq-d\left[V(z)-V(0)\right]\qquad\forall\theta\in I_\delta
        \end{equation}
        with $f(\theta)=(f_1,\cdots,f_N)(\theta)$ defined as 
        $$f_i(\theta)=\frac{K}{N}\sum_{j=1}^Na_{ij}\sin{(\theta_j(t)-\theta_i(t))}\qquad\forall i=1,\ldots,N.$$
        We start by noticing that
        $$\nabla V(\theta)_i=\frac{\theta_i}{\|\theta\|}\qquad\forall i=1,\ldots,N$$
        so that the scalar product becomes
        \begin{align*}
        &\langle\nabla V(\theta),f(\theta)\rangle=\frac{1}{\|\theta\|}\frac{K}{N}\sum_{i=1}^N\theta_i\sum_{j=1}^Na_{ij}\sin(\theta_j-\theta_i)=\frac{1}{\|\theta\|}\frac{K}{2N}2\sum_{i,j=1}^Na_{ij}\theta_i\sin(\theta_j-\theta_i)=\\
        &=\frac{1}{\|\theta\|}\frac{K}{2N}\left(-\sum_{i,j=1}^Na_{ij}\theta_i\sin(\theta_i-\theta_j)+\sum_{i,j=1}^Na_{ij}\theta_i\sin(\theta_j-\theta_i)\right)=\\
        &=\frac{1}{\|\theta\|}\frac{K}{2N}\left(-\sum_{i,j=1}^Na_{ji}\theta_j\sin(\theta_j-\theta_i)-\sum_{i,j=1}^Na_{ij}(-\theta_i)\sin(\theta_j-\theta_i)\right)=\\
        &=-\frac{1}{\|\theta\|}\frac{K}{2N}\sum_{i,j=1}^Na_{ij}(\theta_j-\theta_i)\sin(\theta_j-\theta_i)\leq-\frac{1}{\|\theta\|}\frac{KC_\delta}{2N}\sum_{\substack{i,j=1}}^Na_{ij}(\theta_j-\theta_i)^2
        \end{align*}
        where we have used that
        $$\theta\sin(\theta)\geq \textcolor{black}{C_{\delta}} \theta^2\qquad\forall \theta\in I_{2\delta}$$
        with $\textcolor{black}{C_{\delta}}=\sin(2\delta)/2\delta\in(0,1]$, and the symmetry of $\mathcal{A}$.
        In order to proceed with the proof we would like to rewrite the adjacency matrix $\mathcal{A}$ in a more convenient way. In particular, we want to redefine the diagonal coefficients of $\mathcal{A}$ as 
        $$a_{ii}=-\sum_{j=1,j\neq i}a_{ij}.$$ With an abuse of notation we will still denote with $\mathcal{A}$ the new adjacency matrix of the system.
        This procedure does not change the dynamics of our system because the entries $a_{ii}$ have never played a role in \eqref{RKM} and they can be set freely as we prefer. In this way we have that $-\mathcal{A}$ is a classical Laplacian matrix. In particular we have that
\begin{align*}
    \langle\nabla V(\theta),f(\theta)\rangle\leq-\frac{1}{\|\theta\|}\frac{KC_\delta}{2N}\sum_{\substack{i,j=1}}^Na_{ij}(\theta_j-\theta_i)^2&=-\frac{1}{\|\theta\|}\frac{KC_\delta}{N}\sum_{\substack{i,j=1}}^N-a_{ij}\theta_j\theta_i=\\
    &=-\frac{1}{\|\theta\|}\frac{KC_\delta}{N}\theta^T(-\mathcal{A})\theta
\end{align*}
where we have used that $\sum_i^Na_{ij}=\sum_j^Na_{ij}=0$. Now in order to finish the proof we want to obtain a bound of the form
\begin{equation}\label{lower_bound}
             \theta^\top(-\mathcal{A})\theta\geq\lambda \|\theta\|^2\qquad\text{for some }\lambda>0.
        \end{equation}
        Let us denote with $0=\lambda_1\leq\lambda_2\leq\cdots\leq\lambda_N$ the eigenvalues of $-\mathcal{A}$ arranged in nondecreasing order, counted with their multiplicity. Then since the graph is connected we can see \cite[Lemma 1.7]{chung1997spectral} that we have $\lambda_2>0$. Furthermore from \cite[Theorem 4.2.2]{horn2012matrix} we obtain that
       $$\lambda_2=\min_{\substack{\theta\neq0\\s.t.\theta\perp\underline{1}}}\frac{\theta^\top(-\mathcal{A})\theta}{\|\theta\|^2},$$
       \textcolor{black}{where $\underline{1}=\{1\}^N$ denotes the $N$-dimensional vector whose entries are all equal to $1$}. Due to the symmetries of the system and Lemma \ref{lemma1} we have that both the solutions of \eqref{reduced RKM} and the transformed solutions after the Doss-Sussmann transformation are perpendicular to $\underline{1}$. This is enough to obtain \eqref{lower_bound} and deduce \eqref{Lyapunov_cond} with $d=\frac{KC_\delta\lambda_2}{N}$. The exponential stability follows now from \cite[Theorem 23]{duc2025strong}.\jump\jump
\textit{\textbf{Step 2}}: We are now interested in investigating the rate of convergence $\mu$ towards the synchronization and the size $r(\omega)$ of the random neighborhood $\mathcal{U}=\mathcal{U}(\omega)$ on which the synchronization takes place. In the proof of \cite[Theorem 23]{duc2025strong} it is shown that $\forall t\in[0,\tau]$ 
$$V(y_t)-V(0)\leq\exp\left\{ -t\eta+\frac{L_V\lambda}{\alpha}\right\}\left[V(y_0)-V(0)\right]$$
with $\lambda<\textcolor{black}{\min\left\{\frac{\alpha d}{3L_fL_V},1\right\}}$ so that one has that $\eta\coloneqq d-\frac{1}{\alpha}L_fL_V(2+\lambda)\lambda>0.$ This implies, by induction, that 
$$V(y_1)-V(0)\leq\exp\left\{ -\eta+\frac{L_V\lambda}{\alpha}\textcolor{black}{\overline{\nN}}\left(\frac{\lambda}{16C_pC_{\Tilde{G}}},\textbf{W},[0,1]\right)\right\}\left[V(y_0)-V(0)\right].$$
Following the same argument one obtains $\forall t\in[n,n+1]$ that
\begin{align*}
    V(y_t)-V(0)&\leq\exp\left\{ -\eta(t-k)+\frac{L_V\lambda}{\alpha}\textcolor{black}{\overline{\nN}}\left(\frac{\lambda}{16C_pC_{\Tilde{G}}},\textbf{W},[n,t]\right)\right\}\left[V(y_n)-V(0)\right]\leq\\
    &\leq\exp\left\{ -\eta(t-n)+\frac{L_V\lambda}{\alpha}\textcolor{black}{\overline{\nN}}\left(\frac{\lambda}{16C_pC_{\Tilde{G}}},\textbf{W},[n,n+1]\right)\right\}\left[V(y_n)-V(0)\right].
\end{align*}
So, putting everything together, we obtain $\forall t\in\R^+$, with $t=n+t_0$, $n\in\N$ and $t_0\in[0,1)$ that
\begin{equation}\label{exp_est}
    V(y_t)-V(0)\leq\exp\left\{ -\eta t+\frac{L_V\lambda}{\alpha}\sum_{k=0}^n\textcolor{black}{\overline{\nN}}\left(\frac{\lambda}{16C_pC_{\Tilde{G}}},\textbf{W},[k,k+1]\right)\right\}\left[V(y_0)-V(0)\right].
\end{equation}
We now choose $0<C_{\Tilde{G}}<\lambda<\lambda_0<\textcolor{black}{\min\left\{\frac{\alpha d}{3L_fL_V},1\right\}}$ for sufficiently small $\lambda_0$ such that
$$\Tilde{\mu}\coloneqq d-\frac{1}{\alpha}L_fL_V(2+\lambda_0)\lambda_0-\frac{L_V\lambda_0}{\alpha}\E \textcolor{black}{\overline{\nN}}\left(\frac{1}{16C_p},\textbf{W},[0,1]\right)>0.$$
In order to have stability we ask that $\|y_0\|\leq r(\omega)$ with
$$r(\omega)\coloneqq\epsilon\inf_{n\geq0}\exp\left(\eta n-\frac{L_V\lambda}{\alpha}\sum_{k=0}^n\textcolor{black}{\overline{\nN}}\left(\frac{\lambda}{16C_pC_{\Tilde{G}}},\textbf{W},[k,k+1]\right)\right),$$
for any $\epsilon\in(0,\delta)$. We notice that $r(\omega)>0$ almost surely, indeed we have that
$$\eta n-\frac{L_V\lambda}{\alpha}\sum_{k=0}^n\textcolor{black}{\overline{\nN}}\left(\frac{\lambda}{16C_pC_{\Tilde{G}}},\textbf{W},[k,k+1]\right)\underset{n\to\infty}\longrightarrow+\infty\quad\text{almost surely}.$$
This follows from the Birkhoff's ergodic theorem, so that for all $\varepsilon>0$ there exists $n(\omega)$ such that for all $n\geq n(\omega)$ one has
$$\frac{1}{n}\sum_{k=0}^n\textcolor{black}{\overline{\nN}}\left(\frac{\lambda}{16C_pC_{\Tilde{G}}},\textbf{W},[k,k+1]\right)\leq\E \textcolor{black}{\overline{\nN}}\left(\frac{\lambda}{16C_pC_{\Tilde{G}}},\textbf{W},[0,1]\right)+\varepsilon\leq\E \textcolor{black}{\overline{\nN}}\left(\frac{1}{16C_p},\textbf{W},[0,1]\right)+\varepsilon$$
and consequently, always for $n\geq n(\omega)$, we have that
\begin{align*}
    \eta n-\frac{L_V\lambda}{\alpha}\sum_{k=0}^n\textcolor{black}{\overline{\nN}}\left(\frac{\lambda}{16C_pC_{\Tilde{G}}},\textbf{W},[k,k+1]\right)&\geq n\left(\eta-\frac{L_V\lambda}{\alpha}\E \textcolor{black}{\overline{\nN}}\left(\frac{1}{16C_p},\textbf{W},[0,1]\right)-\frac{L_V\lambda}{\alpha}\varepsilon\right)\\
    &\geq n\left(\Tilde{\mu}-\frac{L_V\lambda_0}{\alpha}\varepsilon\right)=n\mu
\end{align*}
where we have chosen an arbitrary $\varepsilon_0\in(0,\frac{\alpha}{L_V\lambda_0}\Tilde{\mu})$ and defined $\mu\coloneqq\Tilde{\mu}-\frac{L_V\lambda_0}{\alpha}\varepsilon_0>0$. This in particular shows that $r(\omega)>0$ almost surely. With this choice of $r(\omega)$ we obtain that the trivial solution is stable. Now combining \eqref{exp_est} with the previous argument implies that for all $t\in[n,n+1]$ for every $n\geq n(\omega)$ we have
$$V(y_t)-V(0)\leq e^{-\mu t}\left(V(y_0)-V(0)\right).$$
For the other indexes we define
$$R(\omega)\coloneqq\max_{0\leq n<n(\omega)}\exp{\left(\mu(n+1)-\eta n+\frac{L_V\lambda}{\alpha}\sum_{k=0}^n\textcolor{black}{\overline{\nN}}\left(\frac{\lambda}{16C_pC_{\Tilde{G}}},\textbf{W},[k,k+1]\right)\right)}$$
where obviously we have that $1\leq e^\mu\leq R(\omega)<\infty$ and that $\forall t\in[k,k+1]$ with $k+1\leq n(\omega)$
$$V(y_t)-V(0)\leq R(\omega)e^{-\mu t}\left(V(y_0)-V(0)\right).$$ putting the two cases together we obtain that the same inequality holds $\forall t\in\R^+$. Finally, using the fact that $V$ is Lipschitz and \eqref{Lya1} and the previous reasoning one obtains the exponential stability
$$\|y_t\|\leq r^\star(\omega)e^{-\mu t}$$
with $r^\star(\omega)= R(\omega)r(\omega)/\alpha$.
\end{proof}
\begin{rk}
   If the ergodicity assumption on the Wiener shift $\Xi$ is dropped, one can still invoke Birkhoff’s ergodic theorem when estimating stopping times, see \cite[Theorem 4.15]{hawkins2021ergodic}. However, the resulting limits are no longer deterministic constants but random variables. Consequently, the assertions of the main theorem remain valid almost surely, all the stability criteria and parameters would be path dependent. 
\end{rk}
\textcolor{black}{
From the proof above, we observe that $\lambda_0$ must be chosen such that
$$\lambda_0<\min\left\{\frac{KC_\delta\lambda_2}{3NL_f},1\right\},$$
and, moreover, be sufficiently small to ensure that
$$\frac{K}{N}C_\delta\lambda_2-L_f(2+\lambda_0)\lambda_0-\lambda_0\E\textcolor{black}{\overline{\nN}}\left(\frac{1}{16C_p},\textbf{W},[0,1]\right)>0.$$
In particular, these conditions guarantee that, whenever $C_{\Tilde{G}}<\lambda_0$, the quantity appearing as an upper bound for $\mu$ in \eqref{rate_bound} remains strictly positive.} Regarding the random basin of attraction, we have an easy deterministic upper bound that follows directly from its definition, i.e.,
    $$r(\omega)\leq\epsilon$$
    where $\epsilon$ can be chosen arbitrarily in $(0,\delta)$ with $\delta\in[0,\pi/2)$. Of course $\delta$ has to be chosen so that it is bigger than $\max_{ij}\|\theta_i(0)-\theta_j(0)\|$.
    One can give a more precise lower bound by using \eqref{N_up}, but will end up with a random bound as well, up to imposing a priori a deterministic bound on $\tb\textbf{W}\tb_{p-\text{var},[a,b]}$.\\
    
    On the other hand we have managed to give a satisfactory bound on the rate of convergence $\mu$, that involves all the relevant parameters of the system, as how spread are the initial data, encoded by $C_\delta$, the intensity of the noise $C_{\Tilde{G}}$, \textcolor{black}{the coupling strength $K$, the Lipschitz constant of the deterministic drift $L_f$,} the real rough part $\textbf{W}$ of the system and, last but not least, the interactions between the particles encoded by $\lambda_2$. \textcolor{black}{This theoretical bound is consistent with physical intuition, and we now explain why by analyzing each term in the upper bound.}

    \textcolor{black}{
   The first term depends only on the deterministic structure of the model and is the one that contributes positively to synchronization. Indeed, increasing $K$ allows for potentially faster synchronization. The same effect can be achieved by considering highly concentrated initial data, so that $\delta$ is smaller and $C_\delta=\sin(2\delta)/(2\delta)$ becomes larger. Faster synchronization is also promoted by increasing the connectivity of the underlying graph, thereby making $\lambda_2$ larger, as we shall discuss later.}

   \textcolor{black}{
   The last two terms involve the noisy component of the system and therefore it is less obvious whether they act for, or against, synchronization. In particular, the first of these terms accounts for the interplay between the noise intensity $C_{\Tilde{G}}$ and the Lipschitz constant of the deterministic drift $L_f$. In particular, increasing $C_{\Tilde{G}}$ slows down the synchronization process. The fact that the constant $L_f$ amplifies this effect can be interpreted as a magnification of the small perturbations induced by the noise: the more sensitive the deterministic dynamics are, the stronger the impact of those perturbations on the overall behaviour of the system.
   The final term arises from the purely stochastic contribution. Here, the quantity $\E\textcolor{black}{\overline{\nN}}(\cdot)$ captures the intrinsic nature of the noise, rather than merely its intensity $C_{\Tilde{G}}$. Indeed, $\E\textcolor{black}{\overline{\nN}}(\cdot)$ is expected to be larger whenever the rough noise $\textbf{W}$ exhibits highly oscillatory behaviour. From a physical perspective, such rapid oscillations naturally hinder the emergence of synchronization and therefore represent an obstacle to its eventual occurrence.}

   \textcolor{black}{
   The previous discussion also clarifies the relationship between $\lambda_0$ and the underlying physical system. Indeed, it quantifies how large $\lambda_0$ may be and, therefore, how much noise the system can sustain while preserving the synchronization mechanism, in terms of the model parameters.
   }
   \textcolor{black}{
   \begin{rk}
   We notice that $L_f$ also depends on $K$, so increasing the latter parameter does not necessarily result in a straightforward improvement of the rate of convergence. Indeed, it can be shown that
    $$\frac{K}{N}C_\delta \lambda_N\leq L_f\leq\frac{K}{N}\lambda_N$$
    where $\lambda_N$ denotes the largest positive eigenvalue of the Laplacian matrix associated with $\mathcal{A}$. Consequently, the combined effect of the first two terms in the upper bound \eqref{rate_bound} is proportional to
    $$\frac{K}{N}\Big(\lambda_2-\lambda_N(2+C_{\Tilde{G}})C_{\Tilde{G}}\Big).$$
    Therefore, whenever $C_{\Tilde{G}}\ll1$, increasing $K$ does indeed lead to a faster synchronization rate.
   \end{rk}}

    \begin{rk}
    The last point we would like to highlight concerns the parameter $\lambda_2$, in particular its dependence on the structure of the underlying graph. This can be seen using the so-called Cheeger inequality and its related Cheeger constant defined as
    $$h(\mathcal{G})\coloneqq\min\left\{\frac{|\partial X|}{|X|}\;\bigg\vert\; X\subseteq V(\mathcal{G})\;,\;0<|X|\leq\frac{1}{2}|V(\mathcal{G})|\right\}.$$
    Intuitively this constant measures the connectivity of the graph $\mathcal{G}$. Computing this constant is in general an NP-hard problem but for simple enough graphs it is possible to explicitly determine it. Graph classes for which this is possible are, e.g., complete graphs, cycles, trees, paths, complete bipartite graphs, Petersen's graphs, and n-dimensional cube graphs; see \cite{mohar1989isoperimetric} for more details. In the same paper many inequalities regarding $h(\mathcal{G})$ are explored among them the Cheeger inequality that is given, for every graph $\mathcal{G}$, as
    $$\frac{h(\mathcal{G})^2}{2\Delta}\leq\lambda_2\leq2h(\mathcal{G})$$
    with $\Delta$ the maximal vertex degree. Actually the lower bound is stronger whenever $\mathcal{G}$ is not a complete graph with 1,~2 or 3 edges, and it is
    $$h(\mathcal{G})\leq\sqrt{\lambda_2(2\Delta-\lambda_2)}.$$
    These inequalities can be easily generalized for adjacency matrices with real entries, as we are considering in our case. With the help of these inequalities one can simplify and make more explicit the upper bound on the convergence rate $\mu$.    
    \end{rk}
    We now state a Theorem regarding the synchronization of the initial system \eqref{RKM}, which shows how the long term behavior of the system is ruled by the evolution on the mean phase $\Theta$.
    \begin{theorem}\label{theor_synch}
        Under the hypotheses \nameref{A}, \nameref{B}, \nameref{W}, \nameref{G} then there exists a constant $\lambda_0$ such that for any $C_{G}<\frac{\lambda_0}{2}\frac{N}{N-1}$ we have that almost surely, for initial data satisfying
        \begin{equation}\label{hyp_cor_1}
            \underset{i,j}{\max}\|\theta_i(0)-\theta_j(0)\|\leq2\delta\qquad\text{for some $\delta\in[0,\pi/2)$}
        \end{equation}
        \begin{equation}\label{hyp_cor_2}
            \underset{i}{\max}\bigg\|\theta_i(0)-\frac{1}{N}\sum_{j=\textcolor{black}{1}}^N\theta_j(0)\bigg\|\leq r(\omega)
        \end{equation}
    with $r(\omega)$ as in Theorem \ref{main_th}, it holds 
    $$\lim_{t\rightarrow\infty}\|\theta_i(t,\omega)-\Theta(t,\omega)\|=0,\qquad\text{for every }i=1,\ldots,N.$$
    Moreover if there exists $\tilde{p}\in\left(\frac{1}{\gamma},p\right)$ such that
    \begin{equation}\label{moment_est}
        \E\left(\textcolor{black}{\overline{\nN}}_{\tilde{p}}\left(\frac{1}{16C_{\tilde{p}}},\textbf{W},[0,1]\right)^{\frac{2(\tilde{p}+p)-1}{p}}\right)<+\infty
    \end{equation}
    then we have that 
    $$\lim_{t\rightarrow\infty}\|\Theta(t,\omega)-\Theta_\infty(\omega)\|=0\qquad\text{almost surely},$$
    where
    $$\Theta_\infty(\omega)\coloneqq\frac{1}{N}\sum_{i=1}^N\theta_i(0)+\frac{1}{N}\int_0^\infty\sum_{i=1}^N\sum_{j=1}^mG_{ij}(\theta(s,\omega))\d\textbf{W}^j_s.$$
    \end{theorem}
    \begin{proof}
        We first notice that, taking $\lambda_0$ from Theorem \ref{main_th}, if $C_G<\frac{\lambda_0}{2}\frac{N}{N-1}$ then it follows that $C_{\tilde{G}} < \lambda_0$. This follows from the fact that
        \begin{align*}
            \tilde{G}_{ij}=G_{ij}-\frac{1}{N}G_{ij}-\frac{1}{N}\sum_{\underset{k\neq i}{k=1}}^NG_{kj}\leq\left(1-\frac{1}{N}\right)\|G\|_\infty+\frac{N-1}{N}\|G\|_\infty=2\left(1-\frac{1}{N}\right)\|G\|_\infty.
        \end{align*}
        The other hypotheses, including the ones regarding the initial data, ensure that almost surely
        $$\Big\|\hat{\theta}_i(t,\omega)\Big\|\underset{t\rightarrow\infty}{\longrightarrow}0\qquad\text{ for every $i=1,\ldots,N$},$$
        so that, by the definition of $\hat{\theta}$ in \eqref{change_coord}, we have
        $$\Big\|\theta_i(t,\omega)-\Theta(t,\omega)\Big\|\underset{t\rightarrow\infty}{\longrightarrow}0\qquad\text{ for every $i=1,\ldots,N$}.$$
        Moreover, due to \eqref{mean_phase}, we have that 
        $$\Theta(t,\omega)=\Theta(0)+\frac{1}{N}\int_0^t\sum_{i=1}^N\sum_{j=1}^mG_{ij}(\theta(s,\omega))\d\textbf{W}^j_s$$
        so we are left with justifying that, almost surely, we have
        $$\Big\|\Theta(t,\omega)-\Theta_\infty(\omega)\Big\|\underset{t\rightarrow\infty}{\longrightarrow}0.$$
        In particular we are going to prove that $\Theta_\infty(\omega)$ exists almost surely, then the latter convergence can be obtained with the same techniques. In order to achieve this we fix $i\in\{1,\ldots,N\}$ and $j\in\{1,\ldots,m\}$ and we consider 
        $$\mathcal{I}_k(\omega)\coloneqq\int_0^{k}G_{ij}(\theta(s,\omega))\d\textbf{W}^j_s,\qquad\forall k\in\N^+$$
        and we prove that almost surely this is a Cauchy sequence in $\R$. Of course once we have this for every fixed index then we obtain the same for the finite summation that we are interested in. In particular we want to estimate 
        $$\left\|\mathcal{I}_{k_2}(\omega)-\mathcal{I}_{k_1}(\omega)\right\|=\left\|\int_{k_1}^{k_2}G_{ij}(\theta(s,\omega))\d\textbf{W}^j_s\right\|$$
        for all $k_1\leq k_2\in\N^+$. In order to bound such a quantity we want to use a similar version of Theorem~\ref{gub_int}, see \cite[Theorem 31]{friz2017general}, which is possible thanks to the sewing lemma, under $p-$var norm as follows 
        \begin{align*}
        \left\|\int_k^{k+1}Y_s\d\textbf{W}_s\right\|&\leq\|Y_k\|\|W_{k,k+1}\|+\|Y'_k\|\|\mathbb{W}_{k,k+1}\|+\\&\hspace{-1cm}+C_p\left(\tb W\tb_{p-\text{var},[k,k+1]}\tb R^Y\tb_{q-\text{var},[k,k+1]}+\tb \mathbb{W}\tb_{q-\text{var},[k,k+1]}\tb Y'\tb_{p-\text{var},[k,k+1]}\right)
        \end{align*}
        with constant $C_p>1$ independent of $\textbf{W}$ and $Y$, and $Y'$ the Gubinelli derivative of $Y$, defined in Appendix \ref{App_gub}. In our case we define 
        $$Y_s\coloneqq G_{ij}(\theta(s,\omega))=G_{ij}(\theta(s,\omega))-G_{ij}(\underline{\Theta}(s,\omega))$$
        where we have denoted with $\underline{\Theta}$ the $N$-dimensional vector with all the components equal to $\Theta$.
        Now let us bound the terms on the right-hand side of the previous inequality. In particular, we can easily obtain that
        $$\|Y_k\|\leq C_G\|\theta(k,\omega)-\Theta(k,\omega)\|\leq C_G2\pi e^{-\mu k}$$
        and, by using \eqref{kolmo} and the stationarity of the increments, we can deduce that
        \begin{equation}\label{finite_exp_1}
      \E\left(\|W_{k,k+1}\|\right)=\E\left(\|W_{0,1}\|\right)\leq\E\left(\|W_{0,1}\|^p\right)^{\frac{1}{p}}\leq C_{1,\gamma}^{\frac{1}{p}},
        \end{equation}
        \begin{equation}\label{finite_exp_2}
            \E\left(\|\mathbb{W}_{k,k+1}\|\right)=\E\left(\|\mathbb{W}_{0,1}\|\right)\leq\E\left(\|\mathbb{W}_{0,1}\|^q\right)^{\frac{1}{q}}\leq C_{1,\gamma}^{\frac{1}{q}}.
        \end{equation}
        We now focus our attention on estimating the Gubinelli derivative $Y'$ and the remainder $R^Y$. 
        By \cite[Lemma 7.3]{friz2014course} we can now estimate $Y'$ as follows
        \begin{align*}
            \|Y'_k\|&=\left\|\big(G_{ij}(\theta(k,\omega))\big)'\right\|\leq\left\|\big(G(\theta(k,\omega))\big)'\right\|=\Big\|\D G(\theta(k,\omega))G(\theta(k,\omega))\Big\|=\\
            &=\Big\|\D G(\theta(k,\omega))G(\theta(k,\omega))-\D G(\theta(k,\omega))G(\Theta(k,\omega))\Big\|\leq C_G\Big\|G(\theta(k,\omega))-G(\Theta(k,\omega))\Big\|\leq\\
            &\leq C_G^2\Big\|\theta(k,\omega))-\Theta(k,\omega)\Big\|\leq 2\pi C_G^2e^{-\mu k},
        \end{align*}
        where we used that $G\in C^3_b$ and \eqref{rot_inv}. Given this last estimate we are now  able to estimate the $p-$var norm of $Y'$. Indeed by definition we have that for every $k\leq s\leq t\leq k+1$
        $$\|Y'_{s,t}\|\leq\tb Y'\tb_{p-\text{var},[s,t]}\leq\tb Y'\tb_{\tilde{p}-\text{var},[s,t]}$$ for every $\tilde{p}\leq p$. Moreover, from the previous computation we infer that
        $$\|Y'_{s,t}\|\leq4\pi C^2_G e^{-\mu k}$$
        which trivially implies that
        $$\|Y'_{s,t}\|^p\leq\textcolor{black}{ \min\left\{  \tb Y'\tb^p_{\tilde{p}-\text{var},[s,t]},4\pi C^2_G e^{-\mu kp}\right\}.}$$
        Then by using that $\min\{a, b\}\leq a^cb^{1-c}$ for every $a,b\in\R^+$ and every $c\in[0,1]$ we obtain that
        $$\|Y'_{s,t}\|^p\leq\tb Y'\tb^{\tilde{p}}_{\tilde{p}-\text{var},[s,t]}(4\pi C^2_G e^{-\mu kp})^{1-\tilde{p}/p}= C\tb Y'\tb^{\tilde{p}}_{\tilde{p}-\text{var},[s,t]} e^{-\mu k(p-\tilde{p})}$$
        with $\tilde{p}\in(1/\gamma,p)$. Therefore, by considering any partition $k=t_1<t_2<\cdots<t_n=k+1$ and using \cite[Ex. 5.11]{friz2010multidimensional} we have that
        $$\sum_{i=1}^{n-1}\|Y'_{t_i,t_{i+1}}\|^p\leq C\tb Y'\tb^{\tilde{p}}_{\tilde{p}-\text{var},[k,k+1]} e^{-\mu k(p-\tilde{p})}$$
        which implies, after taking the supremum over all the partitions of $[k,k+1]$ that
        $$\tb Y'\tb_{p-\text{var},[k,k+1]}\leq C\tb Y'\tb^{\tilde{p}/p}_{\tilde{p}-\text{var},[k,k+1]} e^{-\mu k(1-\tilde{p}/p)}.$$
        We now derive a similar estimate even for the remainder $R^{Y}$, indeed by definition of the remainder we immediately get 
        \begin{align*}
            \left\|R^{Y}_{s,t}\right\|\leq\|Y_{s,t}\|+\|Y'_s\|\|W_{s,t}\|&\leq 4\pi C_Ge^{-\mu k}+2\pi C^2_G e^{-\mu k}\tb W\tb_{p-\text{var},[s,t]}\leq\\
            &\leq Ce^{-\mu k}\left(1+\tb W\tb_{p-\text{var},[k,k+1]}\right)
        \end{align*}
        for every $k\leq s\leq t\leq k+1$. By applying the same interpolation argument as before we have
        \begin{align*}
            \|R^{Y}\|_{q-\text{var},[k,k+1]}&\leq C\tb R^Y\tb^{\tilde{q}/q}_{\tilde{q}-\text{var},[k,k+1]}\left(1+\tb W\tb_{p-\text{var},[k,k+1]}\right)^{1-\tilde{q}/q}e^{-\mu k(1-\tilde{q}/q)}=\\
            &=C\tb R^Y\tb^{\tilde{p}/p}_{\tilde{q}-\text{var},[k,k+1]}\left(1+\tb W\tb_{p-\text{var},[k,k+1]}\right)^{1-\tilde{p}/p}e^{-\mu k(1-\tilde{p}/p)}
        \end{align*}
        with $\tilde{q}=\tilde{p}/2$.~Now, by combining all the previous bounds and using the estimates \eqref{est_g}, \eqref{est_remainder}, and \eqref{sol_bound}, we obtain that
        \begin{align*}
            &\left\|\int_k^{k+1}Y_s\d\textbf{W}_s\right\|\leq\|Y_k\|\|W_{k,k+1}\|+\|Y'_k\|\|\mathbb{W}_{k,k+1}\|+\\&\hspace{2cm}+C_p\left(\tb W\tb_{p-\text{var},[k,k+1]}\tb R^Y\tb_{q-\text{var},[k,k+1]}+\tb\mathbb{W}\tb_{q-\text{var},[k,k+1]^2}\tb Y'\tb_{p-\text{var},[k,k+1]}\right)\leq\\
            &\leq2\pi C_Ge^{-\mu k}\|W_{k,k+1}\|+2\pi C^2_Ge^{-\mu k}\|\mathbb{W}_{k,k+1}\|+C_p\tb W\tb_{p-\text{var},[k,k+1]}C\tb R^Y\tb^{\tilde{p}/p}_{\tilde{q}-\text{var},[k,k+1]^2}\times\\
            &\times\left(1+\tb W\tb_{p-\text{var},[k,k+1]}\right)^{1-\tilde{p}/p}e^{-\mu k(1-\tilde{p}/p)}+\\
            &\hspace{5.9cm}+C_p\tb\mathbb{W}\tb_{q-\text{var},[k,k+1]^2}C\tb Y'\tb^{\tilde{p}/p}_{\tilde{p}-\text{var},[k,k+1]} e^{-\mu k(1-\tilde{p}/p)}\leq\\
            &\leq\hat{C}e^{-\mu k(1-\tilde{p}/p)}\bigg[\|W_{k,k+1}\|+\|\mathbb{W}_{k,k+1}\|+\tb W\tb_{p-\text{var},[k,k+1]}
            \Big(\tb R^\theta\tb^{\tilde{p}/p}_{\tilde{q}-\text{var},[k,k+1]^2}+\\
            &+\tb W\tb^{\tilde{p}/p}_{\tilde{p}-\text{var},[k,k+1]}\tb \theta\tb^{\tilde{p}/p}_{\tilde{p}-\text{var},[k,k+1]}\Big)\left(1+\tb W\tb_{p-\text{var},[k,k+1]}\right)^{1-\tilde{p}/p}+\\
            &\hspace{5cm}+\tb\mathbb{W}\tb_{q-\text{var},[k,k+1]^2}\tb \theta\tb^{\tilde{p}/p}_{\tilde{p}-\text{var},[k,k+1]}\bigg]\eqqcolon\hat{C}e^{-\mu k(1-\tilde{p}/p)}\mathfrak{X}_k(\omega)
        \end{align*}
        where $\hat{C}$ is a constant, that can change its value from line to line, which contains all the scalars appearing from the other terms. Moreover the random variable $\mathfrak{X}_k$ can be bounded as follows:
        \begin{align*}
            &\mathfrak{X}_k\leq\|W_{k,k+1}\|+\|\mathbb{W}_{k,k+1}\|+\tb \theta,R^\theta\tb_{\tilde{p}-\text{var},[k,k+1]}^{\tilde{p}/p}\bigg(\tb W\tb_{p-\text{var},[k,k+1]}+\tb W\tb_{p-\text{var},[k,k+1]}^{2-\tilde{p}/p}+\\
            &+\tb W\tb_{p-\text{var},[k,k+1]}\tb W\tb_{\tilde{p}-\text{var},[k,k+1]}^{\tilde{p}/p}+\tb W\tb_{p-\text{var},[k,k+1]}^{2-\tilde{p}/p}\tb W\tb_{\tilde{p}-\text{var},[k,k+1]}^{\tilde{p}/p}+\tb \mathbb{W}\tb_{q-\text{var},[k,k+1]}\bigg)\leq\\
            &\leq\|W_{k,k+1}\|+\|\mathbb{W}_{k,k+1}\|+\tb \theta,R^\theta\tb_{\tilde{p}-\text{var},[k,k+1]}^{\tilde{p}/p}\bigg(\tb \textbf{W}\tb_{\tilde{p}-\text{var},[k,k+1]}+\tb \textbf{W}\tb_{\tilde{p}-\text{var},[k,k+1]}^{2-\tilde{p}/p}+\\
            &+\tb \textbf{W}\tb_{\tilde{p}-\text{var},[k,k+1]}^{1+\tilde{p}/p}+2\tb \textbf{W}\tb_{\tilde{p}-\text{var},[k,k+1]}^{2}\bigg)\leq \|W_{k,k+1}\|+\|\mathbb{W}_{k,k+1}\|+\\
            &+\left[\left(\|\theta_k\|+\frac{1}{C_{\tilde{p}}}\textcolor{black}{\overline{\nN}}_{\tilde{p}}\left(\mathfrak{C},\textbf{W},[k,k+1]\right)\right)^{\tilde{p}/p}e^{4\tilde{p}/p}\textcolor{black}{\overline{\nN}}_{\tilde{p}}\left(\mathfrak{C},\textbf{W},[k,k+1]\right)^{\frac{\tilde{p}-1}{p}}+\|\theta_k\|^{\tilde{p}/p}\right]\times\\
            &\hspace{1.5cm}\times\bigg(\tb \textbf{W}\tb_{\tilde{p}-\text{var},[k,k+1]}+\tb \textbf{W}\tb_{\tilde{p}-\text{var},[k,k+1]}^{2-\tilde{p}/p}+\tb \textbf{W}\tb_{\tilde{p}-\text{var},[k,k+1]}^{1+\tilde{p}/p}+2\tb \textbf{W}\tb_{\tilde{p}-\text{var},[k,k+1]}^{2}\bigg).
        \end{align*}
        From the last inequality, assumption \eqref{moment_est}, the stationarity of the increments of the noise, \eqref{finite_exp_1}, \eqref{finite_exp_2} and \eqref{N_down} we can deduce that
        $$\E\big(\mathfrak{X}_k(\omega)\big)\leq D\qquad\forall k\in\N.$$
        with $D\in\R^+$. Putting the latter all together we obtain
        $$\E\left(\left\|\int_{k}^{k+1}Y_s\d\textbf{W}^j_s\right\|\right)\leq\hat{C}De^{-\mu k(1-\tilde{p}/p)}$$
        which trivially implies that
        $$\sum_{k=0}^\infty\E\left(\left\|\int_{k}^{k+1}Y_s\d\textbf{W}^j_s\right\|\right)<+\infty.$$
        By using Borel-Cantelli Lemma 
        or Monotone Convergence Theorem, one deduces that
        $$\int_{k}^{k+1}G_{ij}(\theta(\omega,s))\d\textbf{W}^j_s\underset{k\rightarrow\infty}{\longrightarrow}0\qquad\text{almost surely},$$
        which immediately implies that
        $$\left\|\mathcal{I}_{k_2}(\omega)-\mathcal{I}_{k_1}(\omega)\right\|\underset{k_1,k_2\rightarrow\infty}{\longrightarrow}0\qquad\text{almost surely}$$
        so that $\mathcal{I}_k(\omega)$ is a Cauchy sequence for almost every $\omega$. Therefore, we have established that, almost surely,
        $$\Big\|\Theta(t,\omega)-\Theta_\infty(\omega)\Big\|\underset{t\rightarrow\infty}{\longrightarrow}0$$
        and this concludes the proof.
    \end{proof}
    \begin{rk}
        We would like to make a few observations regarding the latter proof. Firstly, in the second part of the proof, the exponential convergence of the phases $\theta_i$ is not required. In fact, any decay rate for which the $(1-\tilde{p}/p)$-th power is summable would be enough to conclude the proof. 
        
       Furthermore, we observe that \eqref{moment_est} is satisfied by fractional Brownian motions $B^H$ with Hurst parameter $1/4 \leq H \leq 1/2$. In fact, \cite[Remark 6.4]{CLL} shows that for this class of noises one has
      $$\mathbb{E}\left(e^{a \textcolor{black}{\overline{\nN}}_{\tilde{p}}\left(C,\mathbf{B}^H,[s,t]\right)}\right) < +\infty \qquad \forall a \in \mathbb{R}^+,$$
      for any $C \in \mathbb{R}^+$, any $s \leq t \in \mathbb{R}$, and any $\tilde{p} \in \left(\frac{1}{H},4\right)$.
    \end{rk}
    A straightforward extension of the previous result is possible under the more general hypothesis
    \begin{description}
        \item[(\textbf{AI})\label{AI}]  We suppose the $\mathcal{A}$ is symmetric with entries $a_{ij}\geq0$.
    \end{description}
That is the case when $\mathcal{A}$ does not represent a connected graph and can therefore be decomposed as
$$\mathcal{A}=\sum_{i=1}^M\mathcal{A}_i$$
for some $M=2,\ldots,N$ with $\mathcal{A}_i$ a matrix representing a connected graph $\mathcal{G}_i=\{V_i,E_i\}$, with indexes in $\mathcal{N}_i$, and such that for every $i\neq j$ we have that $\mathcal{G}_i$ is not connected with $\mathcal{G}_j$, i.e. $a_{kl}=0$ for every $k\in\mathcal{N}_i,l\in\mathcal{N}_j$. In order to completely decouple the evolution of the sub-populations of particles we furthermore define for every $\ell=1\ldots,M$
$$\Tilde{G}^\ell_{ij}(\theta)\coloneqq G_{ij}(\theta)-\frac{1}{|\mathcal{N}_\ell|}\sum_{k\in\mathcal{N}_\ell} G_{kj}(\theta)\qquad\forall i\in\mathcal{N}_\ell,\forall j=1,\ldots,m.$$
This in particular implies that the dynamics described by \eqref{RKM} can be expressed as
\begin{equation}
    \d\theta_i(t)=\left[\varpi_i+\frac{K}{N}\sum_{j=1}^Na_{ij}\sin{(\theta_j(t)-\theta_i(t))}\right]\d t+(\Tilde{G}^\ell(\theta)\d \textbf{W}_t)_i ,\qquad\forall i\in\mathcal{N}_\ell,\forall\ell=1\ldots,M.
\end{equation}
This immediately yields the following corollary.
    \begin{corollary}
        Under the hypotheses \nameref{AI}, \nameref{B}, \nameref{W}, \nameref{G} then there exists a constant $\lambda_0$ such that for any $C_{G}<\frac{\lambda_0}{2}\frac{N}{N-1}$ we have that almost surely, for initial data satisfying
    $$\underset{i,j\in\mathcal{N}_\ell}{\max}\|\theta_i(0)-\theta_j(0)\|\leq2\delta\qquad\text{for some $\delta\in[0,\pi/2)$},\;\forall\ell=1,\ldots,M$$
    $$\underset{i\in\mathcal{N}_\ell}{\max}\bigg\|\theta_i(0)-\frac{1}{|\mathcal{N}_\ell|}\sum_{j\in\mathcal{N}_\ell}\theta_j(0)\bigg\|\leq r_\ell(\omega)\qquad\forall\ell=1,\ldots,M$$
    with $r_\ell(\omega)$ the analogous as in Theorem \ref{main_th} for each subsystem, \textcolor{black}{with $\textbf{W}$ satisfying \eqref{moment_est},} it holds 
    $$\lim_{t\rightarrow\infty}\|\theta_i(t,\omega)-\Theta^\ell_\infty(\omega)\|=0,\qquad\text{for every }i\in\mathcal{N}_\ell,\forall\ell=1,\ldots,M$$
    with
    $$\Theta^\ell_\infty(\omega)\coloneqq\frac{1}{|\mathcal{N}_\ell|}\sum_{i\in\mathcal{N}_\ell}\theta_i(0)+\frac{1}{|\mathcal{N}_\ell|}\int_0^\infty\sum_{i\in\mathcal{N}_\ell}\sum_{j=1}^mG_{ij}(\theta(s,\omega))\d\textbf{W}^j_s.$$
    \end{corollary}
    \begin{proof}
    It follows directly by applying Theorem \ref{theor_synch} to each sub-population of the initial system \eqref{RKM}.
\end{proof}
\end{subsection}
\begin{subsection}{Splitting for signed graphs}\label{subsec_split}
Thus far, we have considered only particles coupled via positive weights. This models the situation in which all particles agree on a single decision, leading naturally to complete synchronization. Once negative weights $a_{ij}$ are allowed, however, the behavior changes significantly, resulting in mixed regimes where the evolution of the particles is difficult to predict. The sole exception we have identified is when the weights are structured so as to produce a complete splitting of the phases: some particles' phases $\theta_i(t)$ converge to $\theta^\star$, while the remaining converge to $\theta^\star + \pi$.\\

It turns out that the graphs for which this splitting occurs are precisely balanced graphs. First defined in \cite{harary1953notion}, a graph $\mathcal{G}$ is balanced precisely when its vertex set $V$ can be partitioned into two disjoint subsets $E_1$ and $E_2$ such that each positive edge connects vertices within the same subset, and each negative edge connects vertices from different subsets. In the sequel, we let $\mathcal{N}_1$ and $\mathcal{N}_2$ denote the index sets corresponding to $E_1$ and $E_2$, respectively. We now impose the following condition on the adjacency matrix:
\begin{description}
    \item[(\textbf{AII})\label{AII}] We suppose the $\mathcal{A}$ is symmetric and that its associated graph $\mathcal{G}$ is connected and balanced.
\end{description}
Or, following the approach used earlier in order to extend the analysis to the case of a more general disconnected graph:
\begin{description}
    \item[(\textbf{AIII})\label{AIII}] We suppose the $\mathcal{A}$ is symmetric and that its associated graph $\mathcal{G}$ is balanced.
\end{description}
By performing a change of variables we can use again the result developed previously. Namely we define for all $i=1,\ldots,N$
$$\theta_i\longmapsto\varphi_i=\begin{cases}
    \theta_i+\pi\quad&\text{if }i\in\mathcal{N}_1\\
    \theta_i\quad&\text{if }i\in\mathcal{N}_2
\end{cases}.$$
After this change of coordinates it is clear that the variables $\varphi_i$ are governed by 
\begin{equation}
    \d\varphi_i(t)=\left[\varpi_i+\frac{K}{N}\sum_{j=1}^Nc_{ij}\sin{(\varphi_j(t)-\varphi_i(t))}\right]\d t+(G(\varphi)\d \textbf{W}_t)_i ,\qquad\forall i=1,...,N
\end{equation}
where the adjacency matrix $\mathcal{C}$ is symmetric, represent a connected graph and $(\mathcal{C})_{ij}=c_{ij}\geq0$ and $|c_{ij}|=|a_{ij}|$. So we quickly derive another corollary of our main theorem.
\begin{corollary}
        Under the hypotheses \nameref{AII}, \nameref{B}, \nameref{W}, \nameref{G} then there exists a constant $\lambda_0$ such that for any $C_{G}<\frac{\lambda_0}{2}\frac{N}{N-1}$ we have that almost surely, for initial data satisfying \eqref{hyp_cor_1} and \eqref{hyp_cor_2} for the variables $\varphi$, \textcolor{black}{with $\textbf{W}$ satisfying \eqref{moment_est},} it holds 
        $$\lim_{t\rightarrow+\infty}\theta_i(t,\omega)=\begin{cases}
    \Theta_\infty(\omega)-\pi\quad&\text{if }i\in\mathcal{N}_1\\
    \Theta_\infty(\omega)\quad&\text{if }i\in\mathcal{N}_2
\end{cases}\qquad\forall i=1,\ldots,N;$$
with $\Theta_\infty(\omega)$ defined as in Theorem \ref{theor_synch}. Analogous result holds under the hypotheses \nameref{AIII}, \nameref{B}, \nameref{W}, \nameref{G}, with the adequate adjustments given by each sub-population.
\end{corollary}
Once the notion of a balanced graph has been introduced, a natural question arises: how common are such graphs? \textcolor{black}{From a physical perspective they have been employed to describe phenomena arising in social psychology; see, for instance, see \cite{harary1953graph,aref2019balance}.} Unfortunately, \textcolor{black}{in mathematics} balanced graphs are relatively rare compared to the vast class of general signed graphs. An exact, though implicit, enumeration was provided in \cite{harary1981counting}. A more powerful result in this direction is given in \cite[Theorem 2.2-5.1]{el2012balance} and \cite[Theorem 2.3]{el2015weak}. \textcolor{black}{There, in contrast with our model,} random signed graphs denoted by $\mathcal{G}_{N,p_1,q_1}$ are studied. In \textcolor{black}{that} model, for each $i,j=1,\ldots,N$, the weight $a_{ij}$ is positive with probability $p_1$, negative with probability $q_1$, and zero with probability $1-p_1-q_1$. For fixed $p_1,q_1$ satisfying $p_1+q_1\in(0,1)$, it holds that
$$\lim_{N\rightarrow+\infty}\mathbb{P}\left(\mathcal{G}_{N,p_1,q_1} \text{ is balanced}\right)=0.$$
Moreover if the parameters $p_1,q_1$ depend on $N$ it has been shown that there is a critical threshold that guarantees balancedness. Namely if $p_1(N)+q_1(N)=o(N)$ then 
$$\lim_{N\rightarrow+\infty}\mathbb{P}\left(\mathcal{G}_{N,p_1(N),q_1(N)} \text{ is balanced}\right)=1,$$
but as soon as $p_1(N)+q_1(N)\geq(1+\varepsilon)/N$ and $q_1(N)\geq\delta/N$ then
$$\lim_{N\rightarrow+\infty}\mathbb{P}\left(\mathcal{G}_{N,p_1(N),q_1(N)} \text{ is balanced}\right)=0.$$
This reinforces the intuition that the more connections there are in a graph the more difficult it is to have any kind of predictable regime.
\end{subsection}
\begin{subsection}{Local Synchronization of the frequencies}\label{syn_fre}
In this section, we drop hypothesis \nameref{B} while keeping \nameref{A}, \nameref{W}, and \nameref{G}. The price we pay for this more general setting is that we can no longer prove synchronization of the phases $\theta_i(t)$; instead, we only establish synchronization of their frequencies, usually defined as $\varpi_i(t) \coloneqq \dot{\theta}_i(t)$. However, in our framework this definition is not well posed, since the trajectories $\theta_i(t)$ are not classically differentiable. 

In order to overcome this issue, one should interpret the frequencies in a weaker sense, namely in the sense of distributions. Yet this approach also presents a significant obstacle: to the best of our knowledge, there is no method to rigorously derive an RDE for the distributional frequencies $\varpi_i(t)$ starting from $\eqref{RKM}$ for the phases $\theta_i(t)$. For this reason, we instead define the frequencies $\varpi_i(t)$ as the solutions of
    \begin{equation}\label{RKM_fr}
    \begin{cases}
        \displaystyle
    \d\varpi_i(t)=\left[\frac{K}{N}\sum_{j=1}^Na_{ij}\cos{(\theta_j(t)-\theta_i(t))}(\varpi_j(t)-\varpi_i(t))\right]\d t+\left(G(\varpi)\d \textbf{W}_t\right)_i ,\qquad\forall i=1,...,N\\
    \displaystyle
        \varpi_i(0)=\varpi_i
    \end{cases}
    \end{equation}
    and we couple this system with \eqref{RKM}. In Section \ref{Num_sec}, we discuss the validation of this artificial model and identify when it effectively captures the dynamics of the true distributional frequencies. Heuristically, such a system is justified when viewed as a perturbed version of the deterministic frequency system, namely the case $G \equiv 0$. We now proceed by emulating the techniques used in the previous section. In particular, we define 
        $$\overline{\Omega}(t)\coloneqq\frac{1}{N}\sum_{i=1}^N\varpi_i(t)$$
    so that we again have as in \eqref{mean_phase} that
    \begin{equation}\label{mean_frequency}
    \d\overline{\Omega}(t)=\frac{1}{N}\sum_{i=1}^N\sum_{j=1}^m G_{ij}(\varpi(t))\d\textbf{W}^j_t.
    \end{equation}
    We then continue by performing the usual change of variables
    $$\hat{\varpi}(t)\coloneqq\varpi(t)-\overline{\Omega}(t)=(\varpi_1-\overline{\Omega},\ldots,\varpi_N-\overline{\Omega})(t)$$ so that these new variables satisfy the following system
    \begin{equation}\label{RKMfreq}
\begin{cases}
\displaystyle
    \d\hat{\varpi}_i(t)=\left[\frac{K}{N}\sum_{j=1}^Na_{ij}\cos{(\theta_j(t)-\theta_i(t))}(\hat{\varpi}_j(t)-\hat{\varpi}_i(t))\right]\d t+\sum_{j=1}^m \Tilde{G}_{ij}(\hat{\varpi}(t))\d \textbf{W}^j_t ,\qquad\forall i\\
    \displaystyle
    \sum_{i=1}^N\hat{\varpi}_i=0
\end{cases}
\end{equation}
We are now in a position to state and prove a version of Theorem \ref{main_th} for the frequencies.
 \begin{theorem}
    Under the hypotheses \nameref{A}, \nameref{W}, \nameref{G} then there exists a constant $\lambda_0$ such that for any $C_{\Tilde{G}}<\lambda_0$ we have that $\hat{\varpi}^\star=0$ is almost surely exponentially attracting in the sense of {Definition} \ref{defStab} on $\R$ whenever
    \begin{equation}\label{hyp_dist}
        \Delta\coloneqq\sup_{t\in\R}\max_{i,j}|\theta_i(t)-\theta_j(t)|<\frac{\pi}{2}.
    \end{equation}
     The rate of convergence and the basin of attraction can be derived as in Theorem \ref{main_th} and the asymptotic value of $\varpi_i(t)$ is given as in Theorem \ref{theor_synch}.
\end{theorem}
\begin{proof}
We follow the main steps of the proof of Theorem \ref{main_th}. So we consider again
$$V(\varpi)=\|\varpi\|=\sqrt{\sum_{i=1}^N\varpi_i^2}.$$
Of course we have that $V\in C^1(\R\setminus\{0\};\R)\cap C^0(\R;\R)$ and satisfies trivially \eqref{Lya1} and \eqref{Lya2}, with the same constants $L_V=\alpha=\beta=1$. We now find $d>0$ such that \eqref{Lyapunov_cond} holds with $f(\varpi)=(f_1,\cdots,f_N)(\varpi)$ defined as
$$f_i(\varpi)=\frac{K}{N}\sum_{j=1}^Na_{ij}\cos{(\theta_j-\theta}_i)(\varpi_j-\varpi_i).$$
So we now have that 
\begin{align*}
        &\langle\nabla V(\varpi),f(\varpi)\rangle=-\frac{1}{\|\varpi\|}\frac{K}{2N}\sum_{i,j=1}^Na_{ij}\cos{(\theta_j-\theta_i)}(\varpi_j-\varpi_i)^2\\
        &\leq-\frac{1}{\|\varpi\|}\frac{K\Tilde{C}_\delta}{2N}\sum_{\substack{i,j=1}}^Na_{ij}(\varpi_j-\varpi_i)^2=-\frac{1}{\|\varpi\|}\frac{K\Tilde{C}_\delta}{N}\sum_{\substack{i,j=1}}^N-a_{ij}\varpi_j\varpi_i=-\frac{1}{\|\varpi\|}\frac{K\Tilde{C}_\delta}{N}\varpi^T(-\mathcal{A})\varpi
\end{align*}
with $\Tilde{C}=\cos{(\Delta)}$. By using once again the connectivity of $\mathcal{A}$ we can conclude that 
$$\langle\nabla V(\varpi),f(\varpi)\rangle\leq-d\|\varpi\|$$
with $d=\frac{K\Tilde{C}_\delta\lambda_2}{N}$.
\end{proof}
\end{subsection}
We notice that hypothesis \eqref{hyp_dist} is clearly satisfied if we assume \nameref{B} on system \eqref{RKM}. Furthermore, in Section \ref{Num_sec} we show with simulations how our proposed model \eqref{RKM_fr} is capable of capturing the long time behavior of the real distributional frequencies $\dot{\theta}_i$.

\textcolor{black}{At first glance, this result appears to be independent of the maximal pairwise distance between the initial natural frequencies $\Tilde{\Delta}\coloneqq \max_{i,j}|\varpi_i-\varpi_j|.$ This is indeed the case for the constant $\lambda_0$, which, as can be seen from the proof of the theorem, is independent of $\widetilde{\Delta}$. On the other hand, $\Tilde{\Delta}$ influences the validity of assumption \eqref{hyp_dist}. Indeed, numerical simulations suggest that, for fixed coupling strength $K$ and noise intensity, condition \eqref{hyp_dist} eventually fails to hold as $\Tilde{\Delta}$ increases.} 
\begin{subsection}{Deterministic Synchronization for particular noise}\label{subsec_det}
In this section we are now interested in showing a particular behavior of the system, assuming that the noise function $G$ has a particular structure. In particular the additional hypothesis that we need is 
\begin{description}
    \item[(\textbf{H}\textsubscript{\textbf{G}}+)\label{G+}] The components of $G(\theta)$ satisfy
   $$\sum_{i=1}^NG_{ij}(\theta)=0\qquad\forall j=1,\ldots,m.$$
\end{description}
A class of functions that satisfies such hypothesis is the one considered already in \eqref{noise}, with $\tilde{n}$ odd. Under this hypothesis we see immediately from \eqref{mean_phase} that
$$\d\Theta(t)=0,$$
so that it becomes a first integral of the system, i.e. a conserved quantity along the trajectories of the solutions. In particular, the results concerning the asymptotic behavior of the solutions established in the previous sections become more precise, and the limit is now a fully determined constant. Indeed, we have
$$\Theta_\infty(\omega) \equiv \Theta_\infty = \frac{1}{N}\sum_{i=1}^N \theta_i(0).$$
In particular, we are able to characterize the associated rough stable manifold, as stated in the following corollary.
\begin{corollary}\label{Cor_man}
    Under the hypotheses \nameref{A}, \nameref{B}, \nameref{W}, \nameref{G}, \nameref{G+} and $DG(0)=D^2G(0)=0$ we have that the local rough stable invariant manifolds of the coherent equilibria can be characterized almost surely as
    $$\left\{\underline{\theta}\in\mathcal{U}\;\Big|\;\theta^\star=\frac{1}{N}\sum_{i=1}^N\theta_i;\;\max_{i,j}|\theta_i-\theta_j|<\pi\right\}\subseteq\Tilde{W}^s_{loc}\left(\omega,\underline{\theta}^\star\right)\subseteq\left\{\underline{\theta}\in\mathbb{T}^N\;\Big|\;\theta^\star=\frac{1}{N}\sum_{i=1}^N\theta_i\right\},$$
    where $\theta^\star\in[-\pi,\pi]$ and $\mathcal{U}(\omega)$ the random neighborhood of $\underline{\theta}^\star\coloneqq(\theta^\star,\cdots,\theta^\star)$ from Theorem \ref{main_th}.
\end{corollary}
\begin{proof}
    The existence of $\Tilde{W}^s_{loc}$ can be deduced from \cite{KuehnNeamtu21}, where they obtain the existence of a rough center manifold as graph of a Lipschitz function $h$. The same techniques can be used for existence of stable/unstable rough manifolds. Linearizing around $\theta^\star=0$ (up to translations), we obtain the linear operator
    $$\mathcal{L}\theta=\frac{K}{N}\mathcal{A}\theta$$
    and since $-\mathcal{A}$ is a Laplacian matrix we know from \cite[Lemma 1.7]{chung1997spectral} that $\mathcal{L}$ has $\lambda=0$ as an eigenvalue with multiplicity equal to the number of connected components and all the other eigenvalues are negative. This ensure the validity of \cite[Assumption 4.1]{KuehnNeamtu21}, and so the existence of $\Tilde{W}^s_{loc}$. This set can be characterized, by definition, as
    $$\Tilde{W}^s_{loc}\left(\omega,\underline{\theta}^\star\right)=\left\{\theta_0\in\mathcal{V}\subseteq\T^N\;\bigg\vert\;\theta_t(\omega,\theta_0)\underset{t\rightarrow\infty}{\textcolor{black}{\longrightarrow}}x^\star\quad\text{exponentially}\right\}.$$
    The inclusion of the $\Tilde{W}^s_{loc}\left(\omega,\underline{\theta}^\star\right)$ in the RHS is trivial, indeed we just use that the phase mean is a constant of the system so that if, given initial conditions $\theta(0)$ such that their evolution $\theta(t)$ convergence towards the coherent state $\underline{\theta}^\star$ then we have
    $$\frac{1}{N}\sum_{i=1}^N\theta_i(0)=\lim_{n\rightarrow+\infty}\frac{1}{N}\sum_{i=1}^N\theta_i(t)=\frac{1}{N}\sum_{i=1}^N\theta^\star=\theta^\star.$$
    The other inclusion is straightforward as well using Theorem \ref{main_th}, we conclude the proof.
\end{proof}
\begin{rk}
    We conjecture that the assumption $\D G(0)=\D^2G(0)=0$ made in \cite{KuehnNeamtu21} can be removed  entailing the existence of local  manifolds in this case as well. This assumption is not required in \cite{varzaneh2023invariant} which impose the existence of a negative Lyapunov exponent for the existence of local stable manifolds. The techniques in \cite{varzaneh2023invariant} rely on the multiplicative ergodic theorem and do not immediately entail the graph structure of the local manifold as in \cite{KuehnNeamtu21}. 
\end{rk}
Something remarkable about this corollary is that we see that \textit{independently from the noise} the stable invariant manifold is contained in a deterministic hyperplane as can be visualized in Figure \ref{fig_0}. Of course as soon as the algebraic condition \nameref{G+} on the function $G$ is violated, as in the more general setting as before, such property is lost.
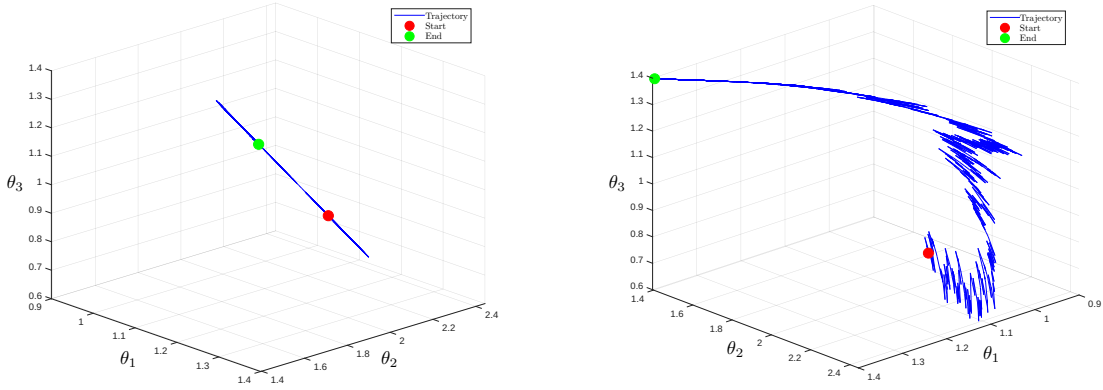
\begin{figure}[H]
  \centering
  \begin{minipage}{0.45\textwidth}
    \centering
    \begin{overpic}[width=\linewidth]{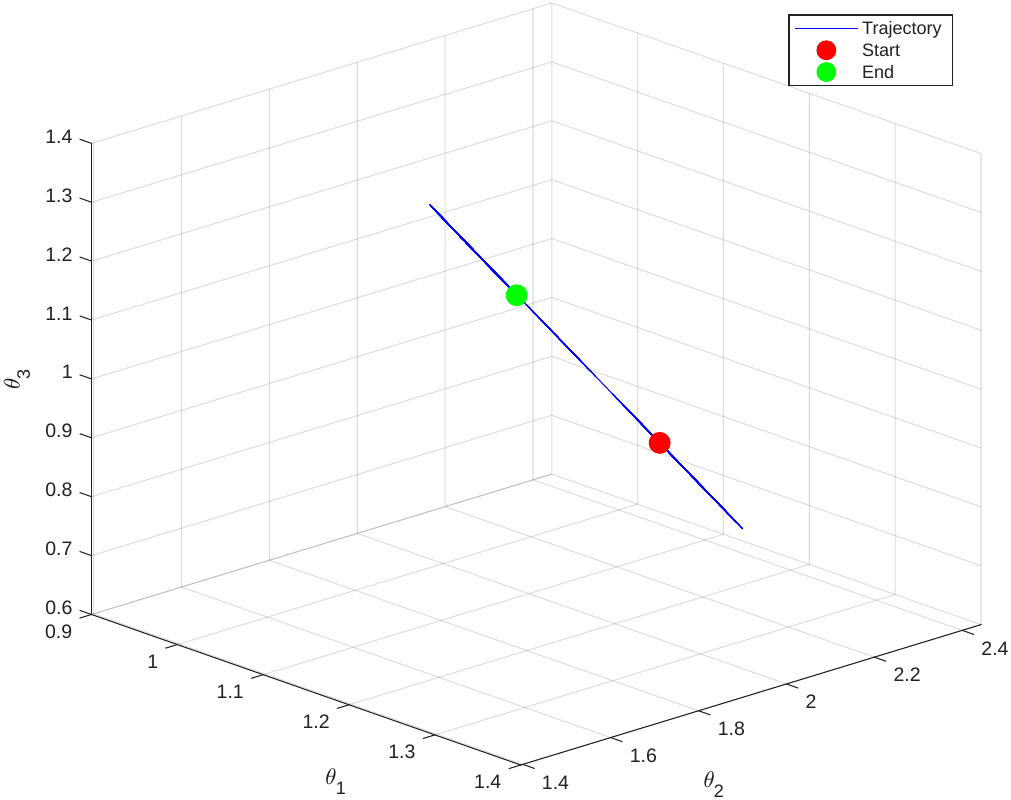}
     
     \put(67,0){\color{white}\rule{0.3cm}{0.4cm}}
     \put(76,5){\scalebox{0.7}{$\theta_2$}}
     
     \put(0,40){\color{white}\rule{0.3cm}{0.4cm}}
     \put(0,41){\scalebox{0.7}{$\theta_3$}}
     
     \put(30,0){\color{white}\rule{0.3cm}{0.4cm}}
     \put(23,5){\scalebox{0.7}{$\theta_1$}}
     
     \put(85,71){\color{white}\rule{0.58cm}{0.4cm}}
     \put(85,75.59){\scalebox{0.35}{Trajectory}}
     \put(85,73.55){\scalebox{0.35}{Start}}
     \put(85,71.55){\scalebox{0.35}{End}}
     \end{overpic}
  \end{minipage}\hfill
  \begin{minipage}{0.45\textwidth}
    \centering
    \begin{overpic}[width=\linewidth]{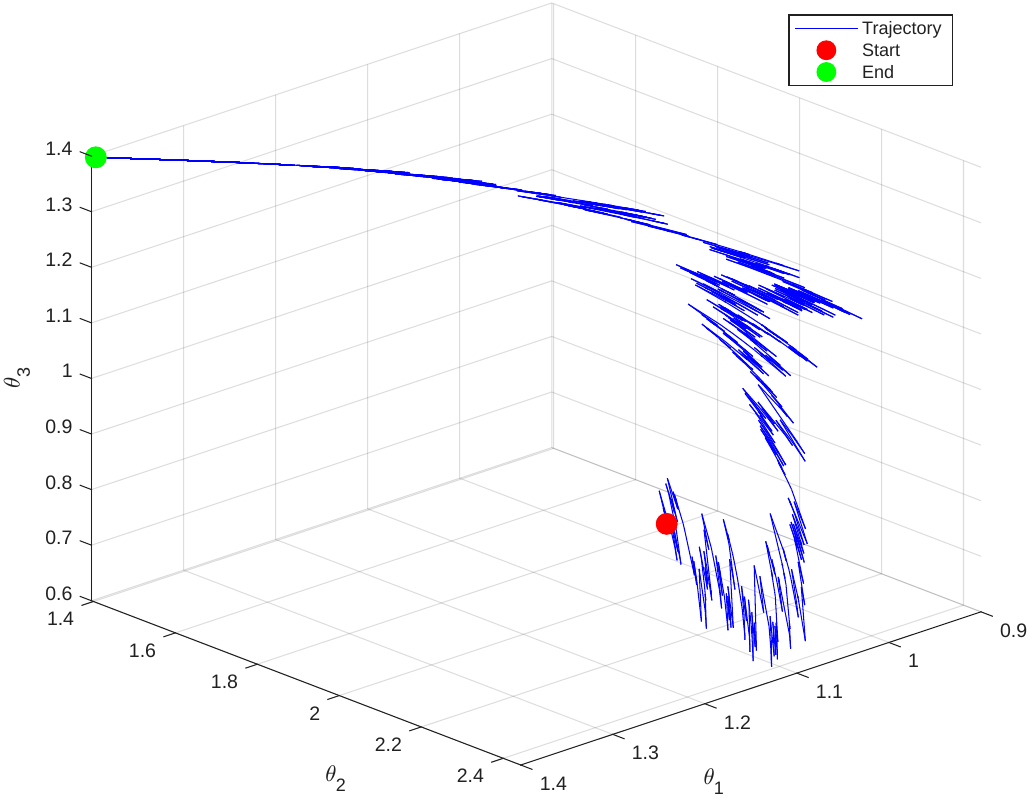}
     
     \put(67,0){\color{white}\rule{0.3cm}{0.4cm}}
     \put(76,5){\scalebox{0.7}{$\theta_1$}}
     
     \put(0,39.5){\color{white}\rule{0.4cm}{0.4cm}}
     \put(0,40){\scalebox{0.7}{$\theta_3$}}
     
     \put(30,0){\color{white}\rule{0.3cm}{0.4cm}}
     \put(24,6){\scalebox{0.7}{$\theta_2$}}
     
     \put(84.0,70){\color{white}\rule{0.54cm}{0.38cm}}
     \put(83.5,74.21){\scalebox{0.35}{Trajectory}}
     \put(83.5,72.17){\scalebox{0.35}{Start}}
     \put(83.5,70.17){\scalebox{0.35}{End}}
     \end{overpic}
  \end{minipage}
  \caption{Simulation of the trajectory $\theta(t) = (\theta_1(t), \theta_2(t), \theta_3(t))$ for the all-to-all coupling case, with initial data satisfying Corollary \ref{Cor_man}, shown from two different perspectives.}
  \label{fig_0}
\end{figure}
Moreover, in the case of different sub-populations, we assume the analogous
\begin{description}
    \item[(\textbf{H}\textsubscript{\textbf{G}}I+)\label{GI+}] The components of $G(\theta)$ satisfy
   $$\sum_{i\in\mathcal{N}_\ell}^NG_{ij}(\theta)=0\qquad\forall j=1,\ldots,m,\;\forall\ell=1,\ldots,M.$$
\end{description}
This hypothesis is satisfied with functions as in \eqref{noise}, with $\mathcal{B}$ that can be decomposed in the same connected components of $\mathcal{A}$ and $\tilde{n}$ odd. In particular we have for free the next general corollary
\begin{corollary}
    Under the hypotheses \nameref{AI}, \nameref{B}, \nameref{W}, \nameref{G}, \nameref{GI+} we have that the local rough stable invariant manifolds of the generalized coherent equilibria can be characterized as
    $$\left\{\underline{\theta}\in\mathcal{U}\;\Big|\;\theta^\star_\ell=\frac{1}{\vert\mathcal{N}_\ell\vert}\sum_{i\in\mathcal{N}_\ell}\theta_i;\;\max_{i,j\in\mathcal{N}_\ell}|\theta_i-\theta_j|<\pi\;\;\forall\ell=1,\ldots m\right\}\subseteq\Tilde{W}^s_{loc}\left(\omega,\underline{\theta}^\star\right),$$
    $$\Tilde{W}^s_{loc}\left(\omega,\underline{\theta}^\star\right)\subseteq\left\{\underline{\theta}\in\mathbb{T}^N\;\Big|\;\theta^\star_\ell=\frac{1}{\vert\mathcal{N}_\ell\vert}\sum_{i\in\mathcal{N}_\ell}\theta_i\;\;\forall\ell=1,\ldots m \right\},$$
    where $\theta^\star\in[-\pi,\pi]$ and $\mathcal{U}$ is a neighborhood of $\underline{\theta}^\star$, with $(\underline{\theta}^\star)_i=\theta^\star_\ell$ for all $i\in\mathcal{N}_\ell$.
\end{corollary}
\begin{proof}
    It follows directly by applying Corollary \ref{Cor_man} to each sub-population of the initial system \eqref{RKM}.
\end{proof}
\end{subsection}
\begin{subsection}{Synchronization without rotational invariance}\label{Syn_rot}
    We now discuss the role of the rotational invariance of $G$, i.e. \eqref{rot_inv}, and what happens if this assumption is removed. Its main mathematical purpose has been to ensure that the system is not affected by the various changes of variables we have performed. We now show how to remove this assumption and which additional conditions are needed in order to still obtain convergence of the phases. In particular, most of the other hypotheses must be adapted with the followings:
    \begin{description}
        \item[(\textbf{H}\textsubscript{\textbf{W}}+)\label{W+}] For every pair of indexes $i,j$ and for any realization $\omega$ we have that $\mathcal{W}^i_t(\omega)=\mathcal{W}^j_t(\omega)$, so that the vector in $\R^m$ has identical components.
        \item[(\textbf{H}\textsubscript{\textbf{G}}II)\label{GII}] $G$ is in $C^3_b(\T^N,\mathcal{L}(\R^m,\T^N))$ and it is a diagonal matrix, where we denote $G_i\coloneqq G_{ii}$. Moreover we suppose that $G(0)=0$ and 
        $$G_i(\underline{a})=G_j(\underline{a})\qquad\forall i,j=1,\ldots,N\text{ and }\forall a\in\R,$$
        where $\underline{a}=(a,\ldots,a)$.
    \end{description}
A class of functions that can be considered in this setting is given by $G_i(\theta)=\sigma \sin(\theta_i)$ or, alternatively, in the case of $\mathbb{R}$, the choice $G_i(\theta)=\sigma \theta_i$ would be natural.  Moreover, under the new hypothesis \nameref{W+}, it is not restrictive to assume that $G$ is diagonal. Indeed, without this assumption, one could simply redefine the noise acting on each oscillator by setting $G_i = \sum_{j=1}^N G_{ij}$. Under these assumptions, after performing the usual change of variables, we obtain the following system:
$$\begin{cases}
\displaystyle
    \d\hat{\theta}_i(t)=\left[\frac{K}{N}\sum_{j=1}^Na_{ij}\sin{(\hat{\theta}_j(t)-\hat{\theta}_i(t))}\right]\d t+  \hat{G}_i(\hat{\theta})\d \textbf{W}_t ,\qquad\forall i=1,...,N\\
    \displaystyle
    \sum_{i=1}^N\hat{\theta}_i=0
\end{cases}$$
with
$$\hat{G}_i(\hat{\theta})\coloneqq\left(G_i(\hat{\theta}+\underline{\Theta}(t)+\underline{\varpi})-\frac{1}{N}\sum_{j=1}^NG_j(\hat{\theta}+\underline{\Theta}(t)+\underline{\varpi})\right),\qquad\sum_{i=1}^N\hat{G}_i(\hat{\theta})=0$$
and $\Theta$ satisfies
$$\d \Theta(t)=\frac{1}{N}\sum_{i=1}^N G_i(\theta+\underline{\varpi})\d\textbf{W}_t.$$
In particular because of our hypothesis \nameref{GII} we have that $\hat{G}(0)=0$. Therefore we are now in the same framework as before and we can state the following corollary.
\begin{corollary}
     Under the hypotheses \nameref{A}, \nameref{B}, \nameref{W}, \nameref{W+}, \nameref{GII} then there exists a constant $\lambda_0$ such that for any $C_{G}<\frac{\lambda_0}{2}\frac{N}{N-1}$ we have that almost surely, for initial data satisfying
        \begin{equation}
            \underset{i,j}{\max}\|\theta_i(0)-\theta_j(0)\|\leq2\delta\qquad\text{for some $\delta\in[0,\pi/2)$}
        \end{equation}
        and
        \begin{equation}
            \underset{i}{\max}\bigg\|\theta_i(0)-\frac{1}{N}\sum_{j=\textcolor{black}{1}}^N\theta_j(0)-\varpi\bigg\|\leq r(\omega)
        \end{equation}
    with $r(\omega)$ as in Theorem \ref{main_th}, it holds 
    $$\lim_{t\rightarrow\infty}\|\theta_i(t,\omega)-\Theta(\omega,t)\|=0,\qquad\text{for every }i=1,\ldots,N.$$
\end{corollary}
The proof of this corollary follows the same lines as that of Theorem \ref{theor_synch}; the only aspect we are unable to establish is the existence of $\Theta_\infty(\omega)$. This makes the case particularly interesting, since the long-time behavior of the phases no longer converges to a limiting random variable, as in the previous setting. Instead, although the trajectories still exhibit convergence, they do not admit a well-defined asymptotic limit, as we will see through our simulations in the next section.
\end{subsection}
\end{section}
\begin{section}{Numerical Simulations}\label{Num_sec}
In this section, we illustrate the previous theoretical results through numerical simulations. Our main goal is to highlight the sharpness of certain hypotheses and to provide evidence that further improvements may be achievable, particularly through a refined theory of Lyapunov functions for rough differential equations. For the simulations, we consider system \eqref{RKM} with
 $$G_i(\theta)=\sigma\sum_{j=1}^Nb_{ij}\sin(\theta_i-\theta_j)\qquad\forall i=1,\ldots,N,\;\sigma\in\R^+$$
as proposed in \eqref{noise}\textcolor{black}{, with $\sigma=0.4$}. \textcolor{black}{In all numerical experiments, we consider a network of $N=10$ oscillators. The adjacency matrices $\mathcal{A}$ and $\mathcal{B}$ are chosen from two classes: the all-to-all coupling case, as defined in Remark~\ref{example_matrix}, and the class of general symmetric matrices. By the latter, we mean matrices satisfying Assumption~\nameref{A}. In the numerical experiments, such matrices are generated randomly as follows: for each pair $i<j$, an entry is independently sampled according to a Bernoulli distribution with parameter $1/2$, and symmetry is then enforced by setting $a_{ij}=a_{ji}$ (and similarly for $\mathcal{B}$).} Moreover, the stochastic process $\mathbf{W}_t$ is chosen to be a fractional Brownian motion with Hurst parameter $H\in\left(\frac{1}{3},\frac{1}{2}\right]$. \textcolor{black}{In particular we set $H=\frac{5}{12}$ and the fBm is generated accordingly to the Davies-Harte algorithm \cite{davies1987tests,wood1994simulation}. All simulations are performed using an explicit Euler-type time discretization with time step $\Delta t=10^{-3}$, driven by the increments of the generated fractional Brownian motion.}  The parameter $K$ is set equal to $1$ in the case of all-to-all coupling, and equal to \textcolor{black}{$N/2=5$} for general symmetric adjacency matrices. This choice is made in order to preserve the rough behavior of the trajectories in the latter case, making it more apparent in the simulations. Indeed, without this scaling, convergence to equilibrium would be proportional to $N$, whereas in the all-to-all coupling case it is typically of order $1$. The initial data are \textcolor{black}{sampled uniformly at random} from the intervals $[0,\pi]$ or $[0,2\pi]$.

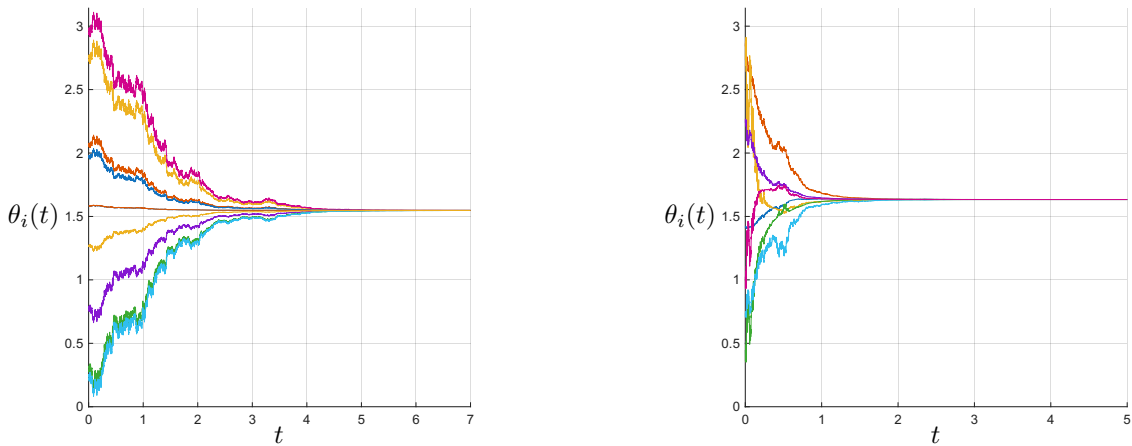
\begin{figure}[H]
  \centering
  \begin{minipage}{0.40\textwidth}
    \centering
    \begin{overpic}[width=\linewidth]{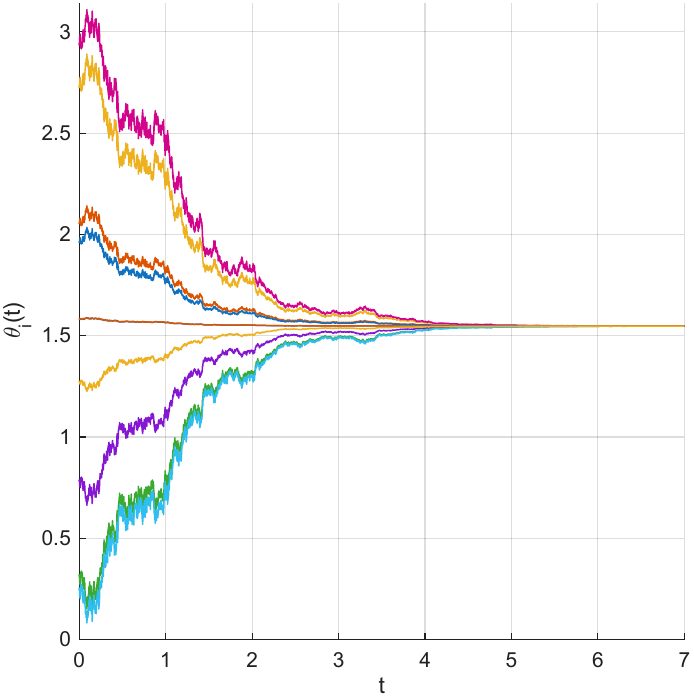}

\put(52,0){\color{white}\rule{0.3cm}{0.3cm}}
\put(53.5,0){$t$}

\put(0,51){\color{white}\rule{0.3cm}{0.4cm}}
\put(-7,50.5){$\theta_i(t)$}
\end{overpic}
  \end{minipage}\hfill
  \begin{minipage}{0.40\textwidth}
    \centering
        \begin{overpic}[width=\linewidth]{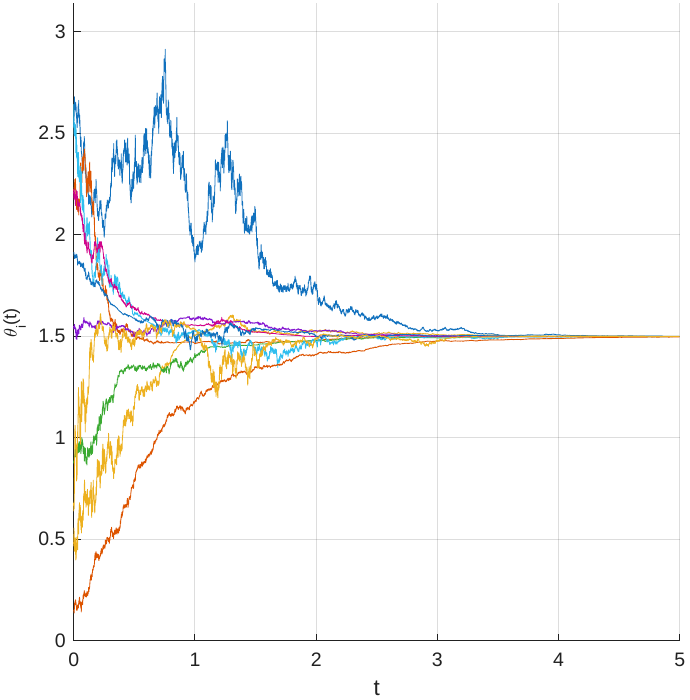}

\put(52,0){\color{white}\rule{0.3cm}{0.3cm}}
\put(53.5,0){$t$}

\put(0,51){\color{white}\rule{0.3cm}{0.4cm}}
\put(-7,50.5){$\theta_i(t)$}
\end{overpic}
  \end{minipage}

  \caption{Plots illustrating phase synchronization in accordance with our theoretical results. The left panel shows the case of all-to-all coupling, while the right panel corresponds to a general symmetric adjacency matrix.}
  \label{fig_1}
\end{figure}
In Figure \ref{fig_1}, we present simulations with initial data satisfying the hypotheses of our previous theoretical results. In particular, the condition on the initial data $\max_{i,j}|\theta_i(0)-\theta_j(0)|<\pi$ appears to be essentially sharp from a theoretical standpoint. This can be readily illustrated by considering the following simple setting: let $\mathcal{A}$ and $\mathcal{B}$ be all-to-all adjacency matrices, and choose the initial data as
$$\theta_i(0)=\begin{cases}
    \theta^\star\quad&\text{if }i=1,\ldots,N_0,\\
    \theta^\star+\pi\quad&\text{if }i=N_0+1,\ldots,N,
\end{cases}$$
for a fixed $\theta^\star\in[-\pi,\pi]$ and a fixed $N_0\in\{1,\ldots,N\}$. Then, clearly, we have $\mathrm{d}\theta_i=0$ for all $i$, and hence no synchronization occurs. 

On the other hand, simulations such as those reported in Figure \ref{fig_2} indicate that synchronization may still occur even when the initial data are distributed over the whole interval $[0,2\pi]$.
\begin{figure}[H]
  \centering
  \begin{minipage}{0.40\textwidth}
    \centering
    \begin{overpic}[width=\linewidth]{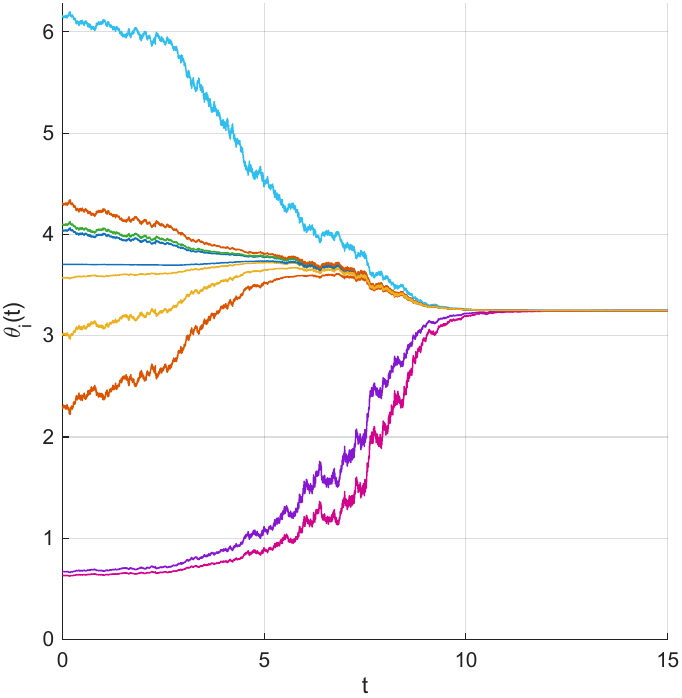}

\put(52,0){\color{white}\rule{0.3cm}{0.3cm}}
\put(53.5,0){$t$}

\put(0,51){\color{white}\rule{0.3cm}{0.4cm}}
\put(-7,50.5){$\theta_i(t)$}
\end{overpic}
  \end{minipage}\hfill
  \begin{minipage}{0.40\textwidth}
    \centering
    \begin{overpic}[width=\linewidth]{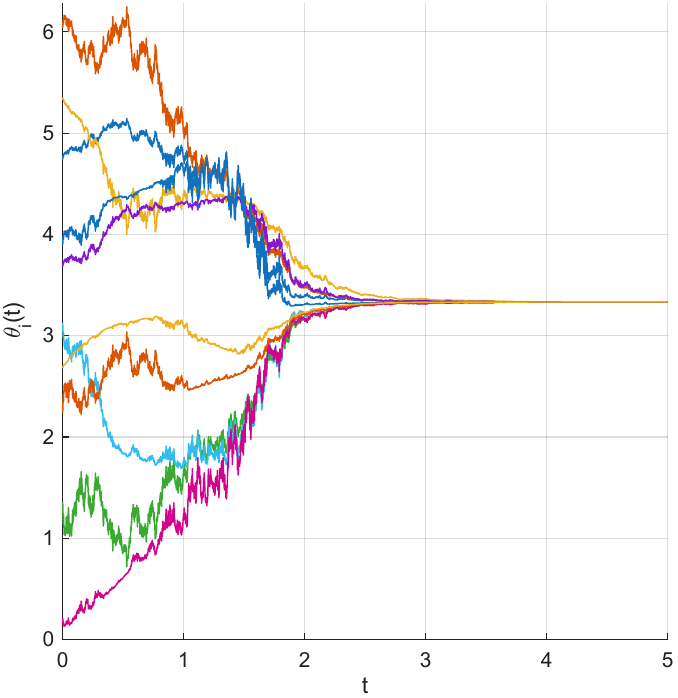}

\put(52,0){\color{white}\rule{0.3cm}{0.3cm}}
\put(53.5,0){$t$}

\put(0,51){\color{white}\rule{0.3cm}{0.4cm}}
\put(-7,50.5){$\theta_i(t)$}
\end{overpic}
  \end{minipage}

  \caption{Plots showing that synchronization occurs even when our hypotheses are violated. The left panel corresponds to the case of all-to-all coupling, while the right panel shows the case of a general symmetric adjacency matrix.}
  \label{fig_2}
\end{figure}
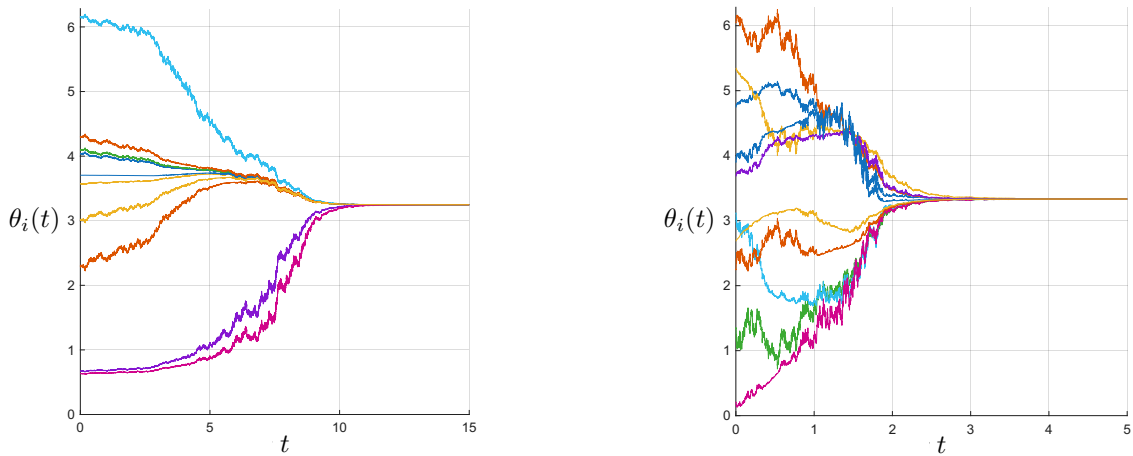
Actually, numerical evidence suggests that global synchronization occurs as soon as there exists at least one pair of particles, indexed by $i_0$ and $j_0$, such that $|\theta_{i_0}(0)-\theta_{j_0}(0)|\neq\pi.$ On the other hand, the intensity of the noise appears to play a crucial role. In particular, when the parameter $\sigma\gg1$, complete synchronization seems to be lost, and a splitting phenomenon, similar to that observed for balanced graphs, emerges, as illustrated in Figure \ref{fig_3}, even when the initial data lie in $[0,\pi]$. This can be explained by the fact that, for $\sigma\gg1$, the condition $C_{\tilde{G}}<\lambda_0$ is no longer satisfied.
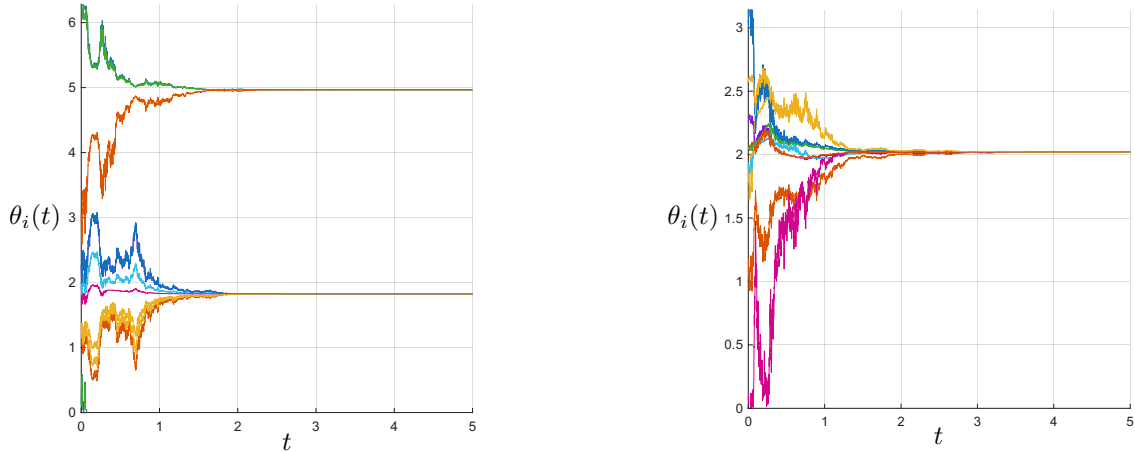
\begin{figure}[H]
  \centering
  \begin{minipage}{0.40\textwidth}
    \centering
    \begin{overpic}[width=\linewidth]{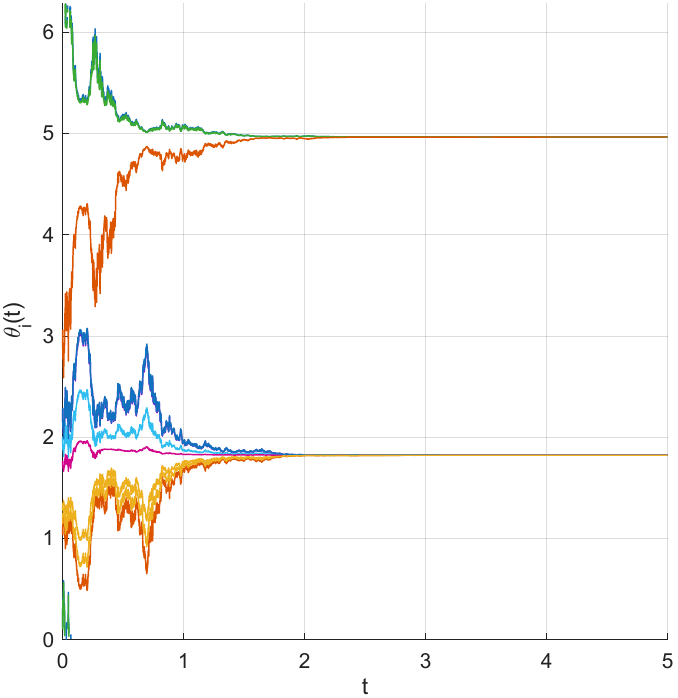}

\put(52,0){\color{white}\rule{0.3cm}{0.3cm}}
\put(53.5,0){$t$}

\put(0,51){\color{white}\rule{0.3cm}{0.4cm}}
\put(-7,50.5){$\theta_i(t)$}
\end{overpic}
  \end{minipage}\hfill
  \begin{minipage}{0.40\textwidth}
    \centering
    \begin{overpic}[width=\linewidth]{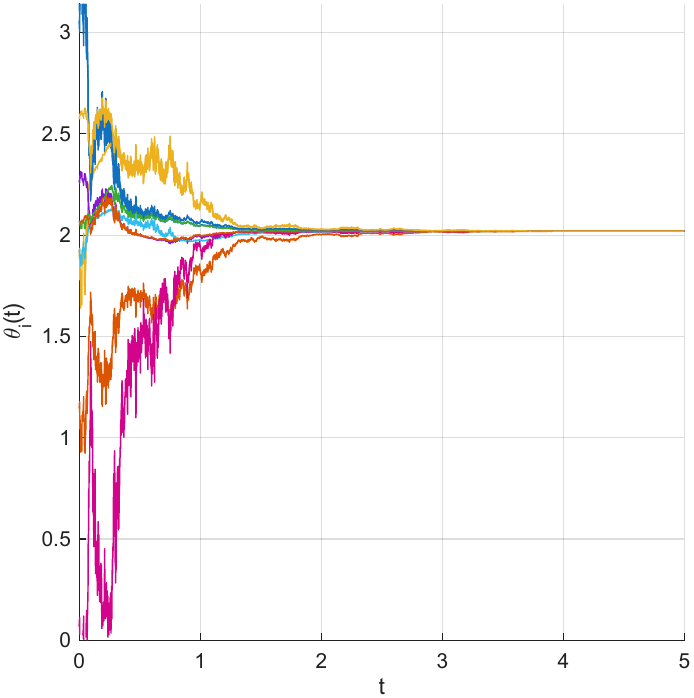}

\put(52,0){\color{white}\rule{0.3cm}{0.3cm}}
\put(53.5,0){$t$}

\put(0,51){\color{white}\rule{0.3cm}{0.4cm}}
\put(-7,50.5){$\theta_i(t)$}
\end{overpic}
  \end{minipage}
  \caption{Left: splitting observed in the all-to-all coupling configuration for $\sigma \gg 1$. Right: synchronization in a disconnected graph, arising from the connected structure of the noise function.}
  \label{fig_3}
\end{figure}
As noted earlier, numerical simulations indicate that phase synchronization can occur even in cases where $\mathcal{A}$ is a disconnected graph, as long as $\mathcal{B}$ is connected. Our expectation is that this phenomenon can be rigorously proven in the future. This would require the development of a more sophisticated Lyapunov function method for rough equations, one that integrates the algebraic effect of the noise into the condition stated in \eqref{Lyapunov_cond}.\jump

Furthermore, as shown in Subsection \ref{Syn_rot}, the occurrence of synchronization does not strictly require rotational invariance. Under appropriate additional conditions on the noise $\textbf{W}$ and the noise function $G$, synchronization is still observed. However, rather than converging to a deterministic fixed point, the phases converge to a random attractor, which results in persistent stochastic oscillations at large times. This behavior is depicted in Figure \ref{fig_4}.
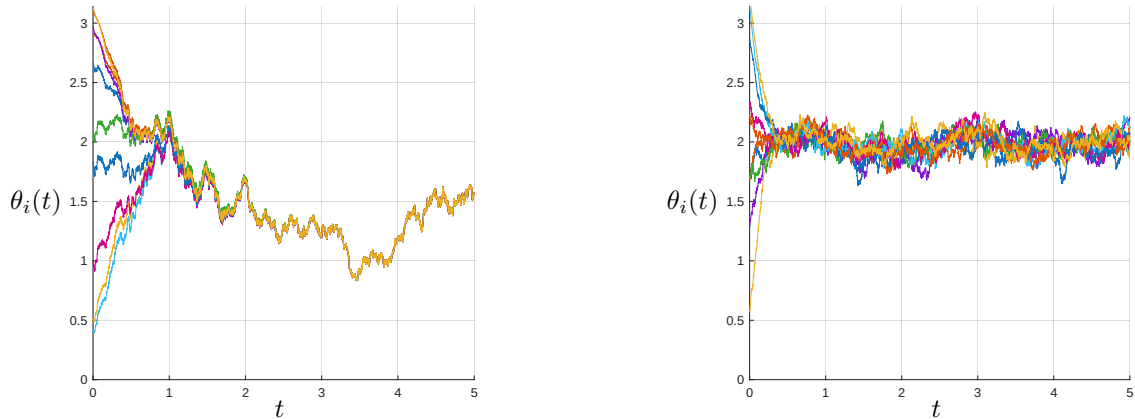
\begin{figure}[H]
  \centering
  \begin{minipage}{0.40\textwidth}
    \centering
        \begin{overpic}[width=\linewidth]{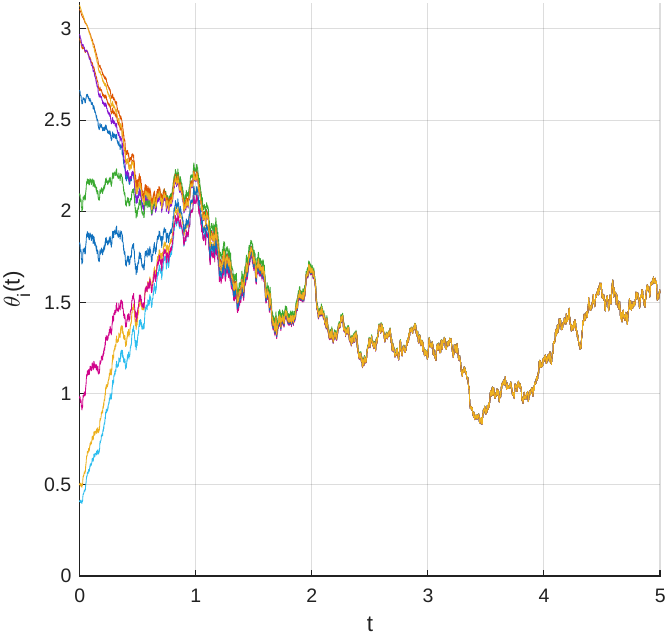}

\put(52,0){\color{white}\rule{0.3cm}{0.3cm}}
\put(53.5,0){$t$}

\put(0,48){\color{white}\rule{0.3cm}{0.4cm}}
\put(-7,48){{$\theta_i(t)$}}
\end{overpic}
  \end{minipage}\hfill
  \begin{minipage}{0.40\textwidth}
    \centering
    \begin{overpic}[width=\linewidth]{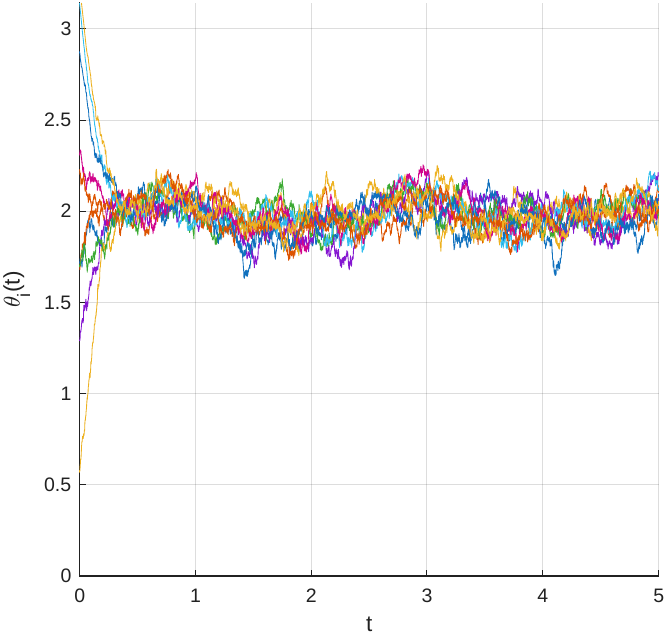}

\put(52,0){\color{white}\rule{0.3cm}{0.3cm}}
\put(53.5,0){$t$}

\put(0,48){\color{white}\rule{0.3cm}{0.4cm}}
\put(-7,48){$\theta_i(t)$}
\end{overpic}
  \end{minipage}
  \caption{Left: synchronization for $G_i(\theta) = \sin(\theta_i)$ with identical noise. Right: desynchronization when the components of $\textbf{W}$ are non-identical.}
  \label{fig_4}
\end{figure}
We now address the local synchronization of the frequencies $\varpi_i$, a concept introduced in Section \ref{syn_fre}. \textcolor{black}{In this case, the initial frequencies are chosen independently and identically distributed according to the Gaussian distribution $\mathcal{N}(0,1)$.} As observed earlier, the quantities $\varpi_i(t)$, when viewed as the system's instantaneous frequencies, generally require interpretation in a distributional sense. The following simulation of \eqref{RKM} demonstrates their typical behavior when the initial frequencies are non-identical. Specifically, we present both the mean of the raw instantaneous frequencies and a smoothed version obtained by averaging over various time intervals. \textcolor{black}{In the following, Figures~\ref{fig_5}--\ref{fig_6} show the results of the same simulation with identical initial data. The purpose of these figures is to compare the evolution of the frequencies extracted from the phase trajectories, $\dot{\theta}(t)$, with the evolution of the auxiliary frequencies $\varpi(t)$ defined by system~\eqref{RKM_fr}.}

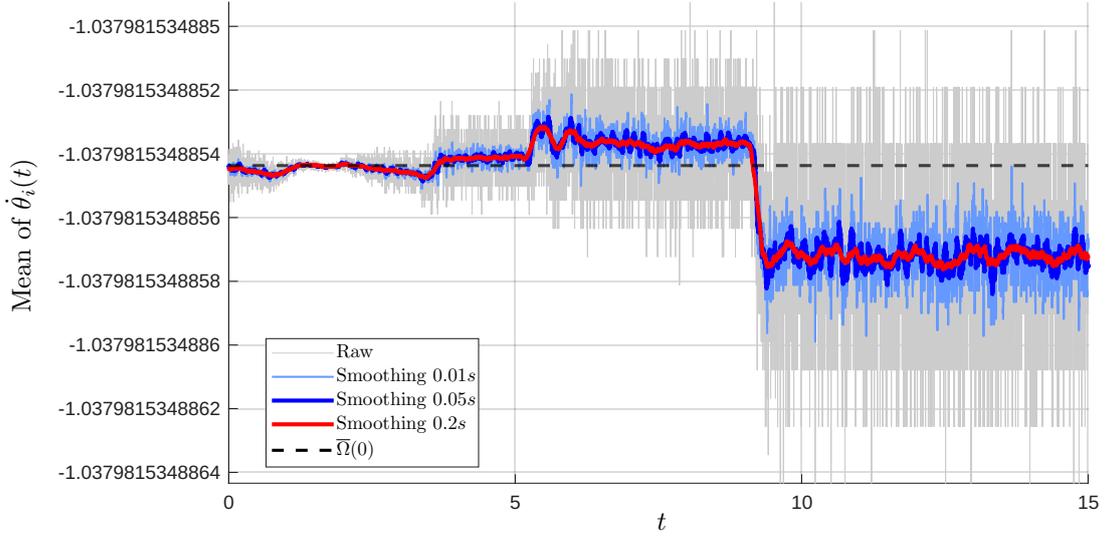
\begin{figure}[H]
  \hspace*{-0.2cm} 
  \begin{overpic}[width=\linewidth]{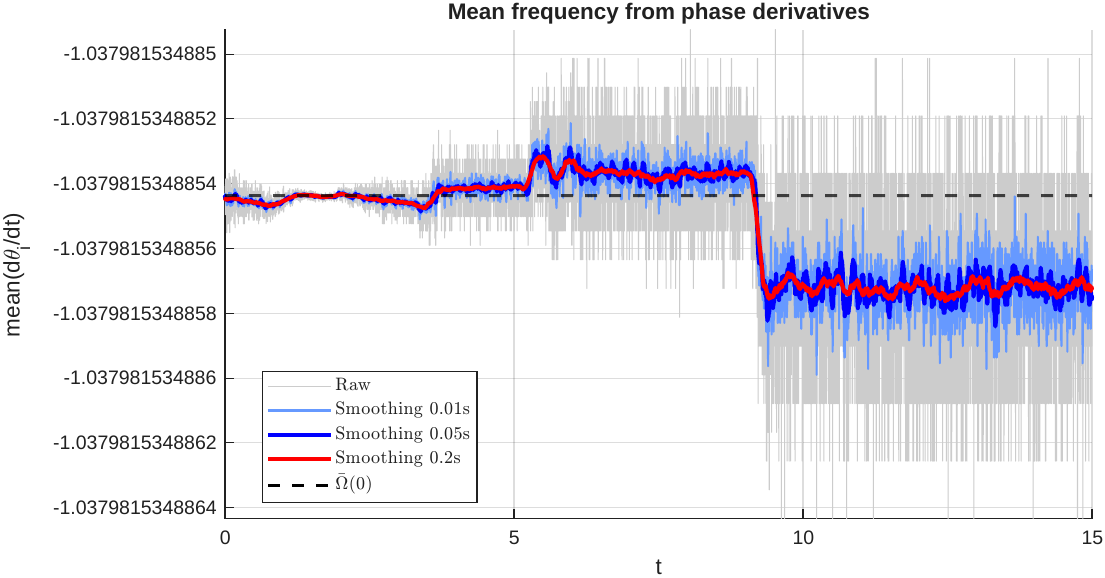}

\put(40,50){\color{white}\rule{6cm}{0.4cm}}
\put(40,0){\color{white}\rule{6cm}{0.4cm}}
\put(0,0){\color{white}\rule{0.4cm}{6cm}}
\put(30.2,7.1){\color{white}\rule{1.8cm}{1.6cm}}
\put(30.2,16.85){\scalebox{0.72}{Raw}}
\put(30.2,14.65){\scalebox{0.72}{Smoothing $0.01s$}}
\put(30.2,12.50){\scalebox{0.72}{Smoothing $0.05s$}}
\put(30.2,10.30){\scalebox{0.72}{Smoothing $0.2s$}}
\put(30.2,7.85){\scalebox{0.72}{$\overline{\Omega}(0)$}}
\put(0,21){\rotatebox{90}{Mean of $\dot{\theta}_i(t)$}}
\put(59.5,1){$t$}
\end{overpic}
  \caption{Simulation of the mean frequency $\overline{\Omega}(t)$, smoothed over different time intervals}
  \label{fig_5}
\end{figure}
\textcolor{black}{In particular, we note that the mean evolution of the frequencies $\dot{\theta}_i$s remains nearly constant. This can be explained by the highly irregular behavior of the individual frequencies, whose fluctuations largely compensate each other when averaged, yielding an almost constant mean.} In contrast, Figure \ref{fig_6} reveals that the individual frequencies $\Dot{\theta}_i(t)$ are far from synchronized, exhibiting highly irregular evolution which is consistent with their interpretation in a distributional sense. However, under condition ($\mathbf{H_GI+}$), Figure \ref{fig_6} further shows that \eqref{RKM_fr} faithfully captures the asymptotic behavior of the original system, up to smoothing. If ($\mathbf{H_GI+}$) does not hold, one would still expect frequency synchronization \textcolor{black}{for the system} \eqref{RKM_fr}, modulo a translation by a random constant, as in Theorem \ref{theor_synch}.
\begin{figure}[H]
  \centering
  \begin{minipage}{0.33\textwidth}
    \centering
    \begin{overpic}[width=\linewidth]{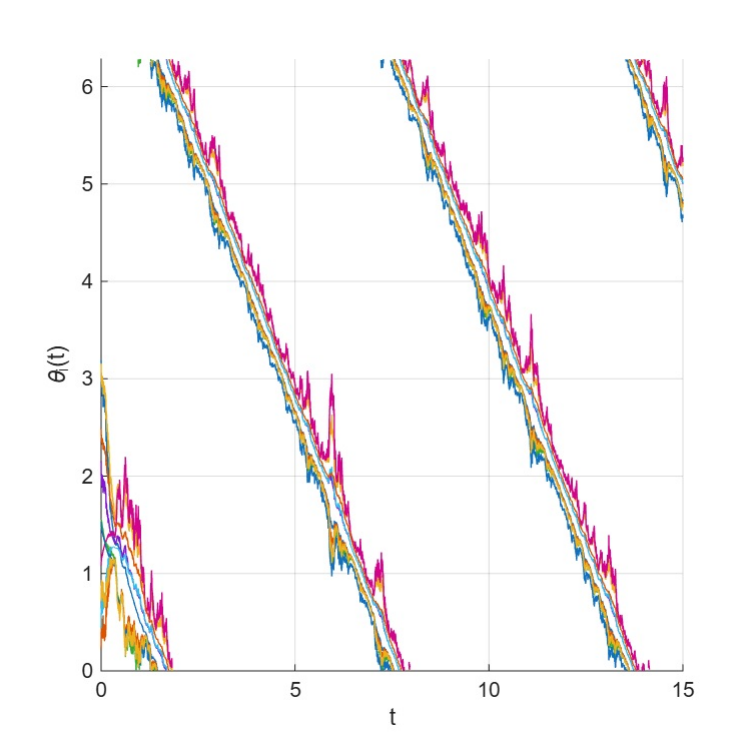}

\put(50,1){\color{white}\rule{0.3cm}{0.3cm}}
\put(51,2){\scalebox{0.8}{$t$}}

\put(4,48){\color{white}\rule{0.3cm}{0.4cm}}
\put(-2,50){\scalebox{0.8}{$\theta_i(t)$}}
\end{overpic}
  \end{minipage}\hfill
  \begin{minipage}{0.33\textwidth}
    \centering
    \begin{overpic}[width=\linewidth]{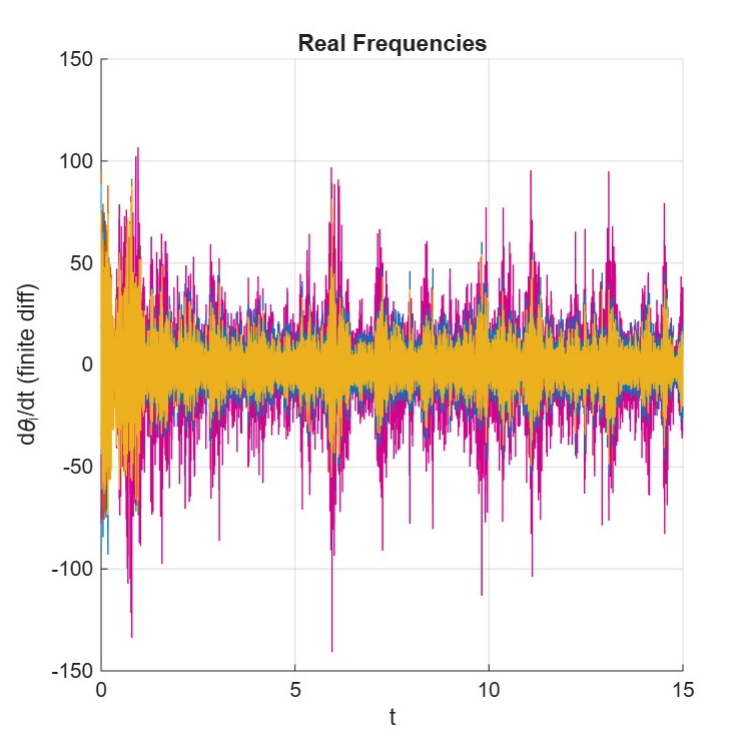}

 \put(35,92.4){\color{white}\rule{3cm}{0.3cm}}
\put(50,1){\color{white}\rule{0.3cm}{0.3cm}}
\put(51,2){\scalebox{0.8}{$t$}}

\put(0,30){\color{white}\rule{0.3cm}{2cm}}
\put(-2,50){\scalebox{0.8}{$\dot{\theta}_i(t)$}}
\end{overpic}
  \end{minipage}\hfill
  \begin{minipage}{0.33\textwidth}
    \centering
    \begin{overpic}[width=\linewidth]{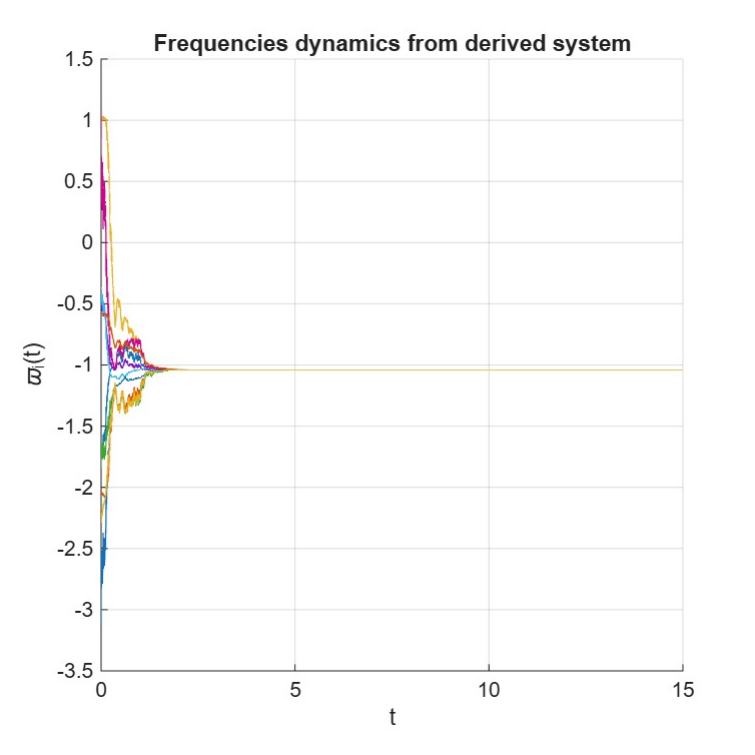}

 \put(20,92.4){\color{white}\rule{4cm}{0.3cm}}
\put(50,1){\color{white}\rule{0.3cm}{0.3cm}}
\put(51,2){\scalebox{0.8}{$t$}}

\put(2,48){\color{white}\rule{0.3cm}{0.4cm}}
\put(-6,50){\scalebox{0.8}{$\varpi_i(t)$}}
\end{overpic}
  \end{minipage}
  \caption{On the left: simulation of the phases of \eqref{RKM} with different frequencies. In the center: simulation of the frequencies obtained as a finite difference of the phases on small time intervals. On the right: simulation of the variables $\varpi_i$ subjected to the system \eqref{RKM_fr}.}
  \label{fig_6}
\end{figure}
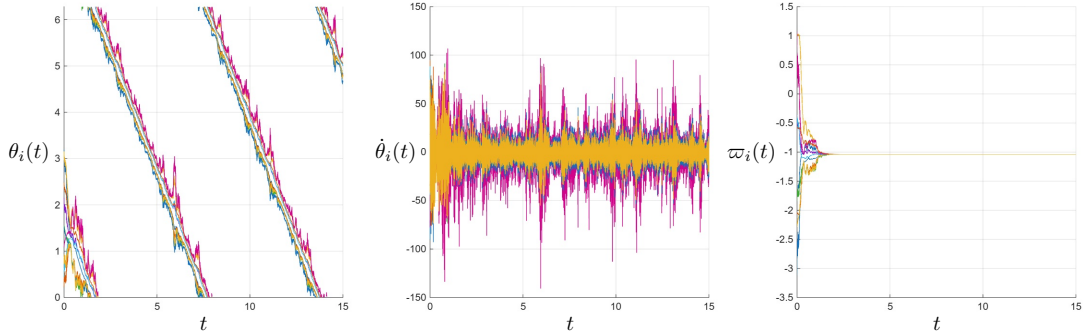
This indicates, in particular, that the analysis of the frequencies in \cite{wu2020global} is not valid and that frequency synchronization cannot be established so easily.\jump

\textcolor{black}{
We now compare the optimal theoretical convergence rate, $\mu^\star$, given by the upper bound \eqref{rate_bound} established in Theorem \ref{main_th}, with the effective convergence rate $\mu_\e$ observed in our simulations. Our goal is to investigate how these quantities vary with the parameters $\sigma$ and $K$, thereby assessing both the influence of these parameters and the sharpness of the theoretical estimate $\mu^\star$.
Figures \ref{fig_7}--\ref{fig_8}--\ref{fig_9} each contain three panels. The left panel shows the theoretical rate $\mu^\star$ as a function of either $\sigma$ or $K$. The central panel displays the mean and variance of the empirical rate $\mu_\e$, computed from $100$ independent simulations. The right panel reports the synchronization probability $\mathbb{P}$, evaluated using the same set of simulations. A simulation is declared successful if the system reaches a synchronization error below $10^{-2}$ by the final time $T$.
For Figures \ref{fig_7} and \ref{fig_8}, we set $T=10$ and fix $K=10$. In Figure \ref{fig_9}, we instead fix $\sigma=0.001$ and take $T=50$.} 

\begin{figure}[H]
  \centering
  \begin{minipage}[t]{0.30\textwidth}
    \centering
    \begin{overpic}[width=\linewidth]{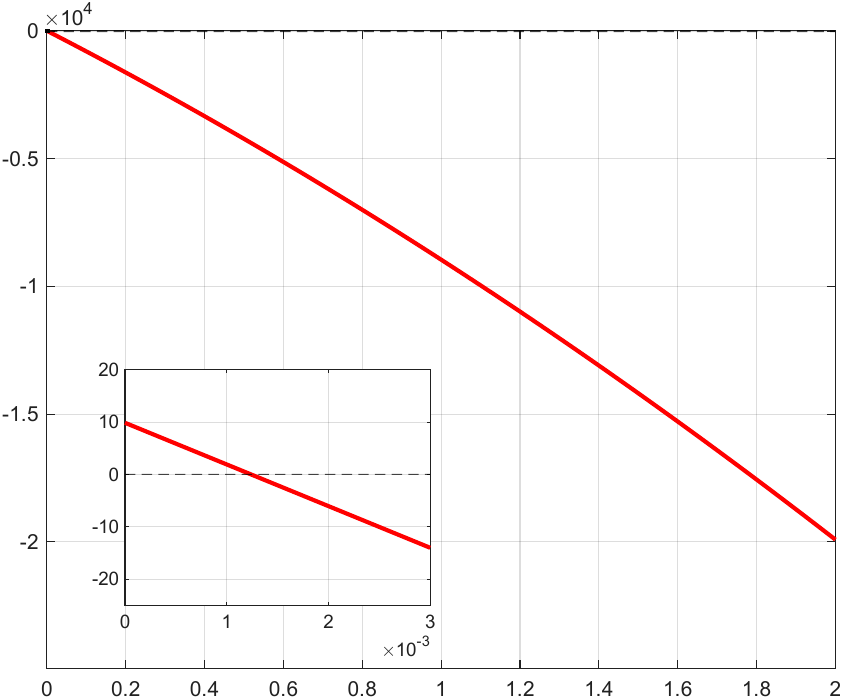}

\put(50,-6){\scalebox{1}{$\sigma$}}

\put(-6,40){\scalebox{1}{$\mu^\star$}}
\end{overpic}
  \end{minipage}\hfill
  \begin{minipage}[t]{0.30\textwidth}
    \centering
    \begin{overpic}[width=\linewidth]{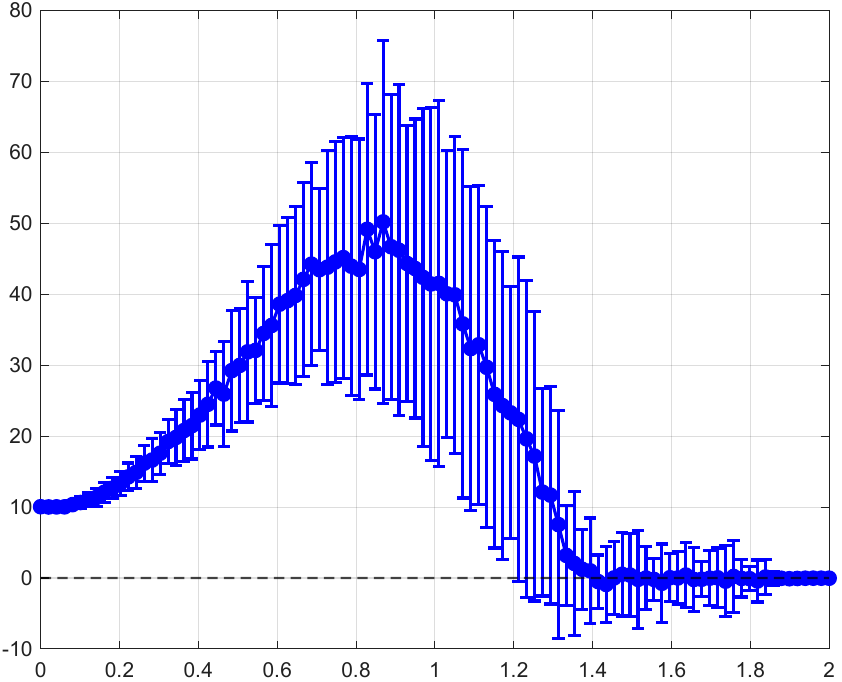}

 \put(49.5,-6){\scalebox{1}{$\sigma$}}

\put(-8,41){\scalebox{1}{$\mu_\e$}}
\end{overpic}
  \end{minipage}\hfill
  \begin{minipage}[t]{0.30\textwidth}
    \centering
    \begin{overpic}[width=\linewidth]{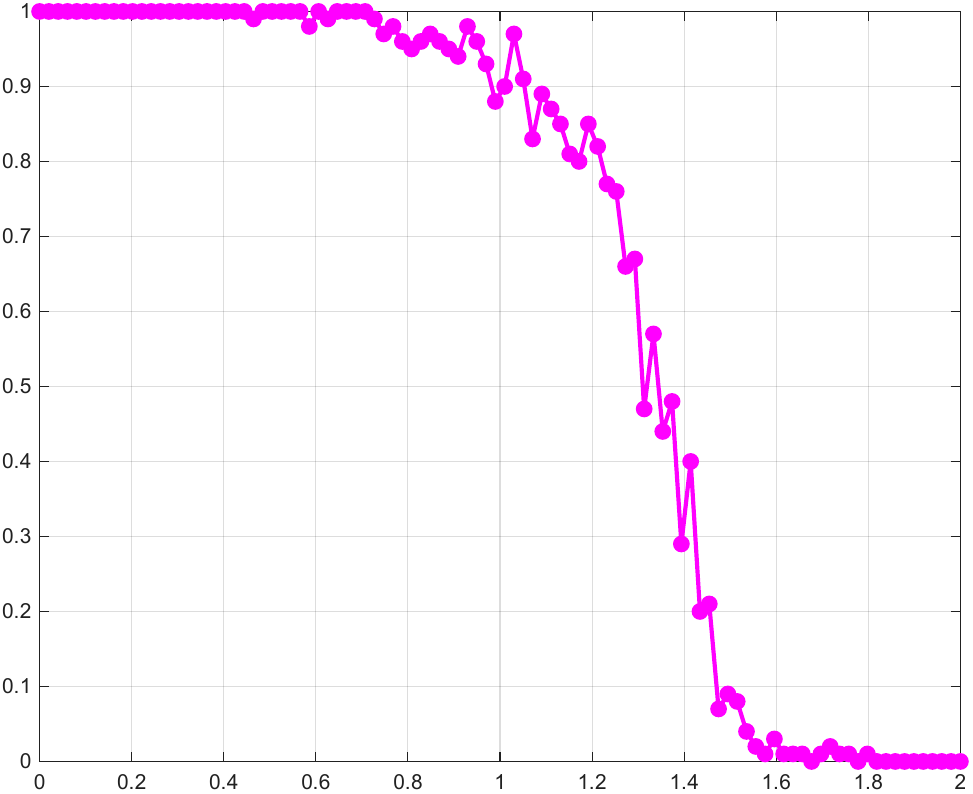}

\put(49,-6){\scalebox{1}{$\sigma$}}

\put(-8,40){\scalebox{1}{$\mathbb{P}$}}
\end{overpic}
  \end{minipage}
  \caption{\textcolor{black}{Plot of $\mu^\star,\mu_\e$ and $\mathbb{P}$ over $\sigma$ of initial data uniformly sampled from $[0,\pi/3]$}}
  \label{fig_7}
\end{figure}
\begin{figure}[H]
  \centering
  \begin{minipage}[t]{0.308\textwidth}
    \centering
    \begin{overpic}[width=\linewidth]{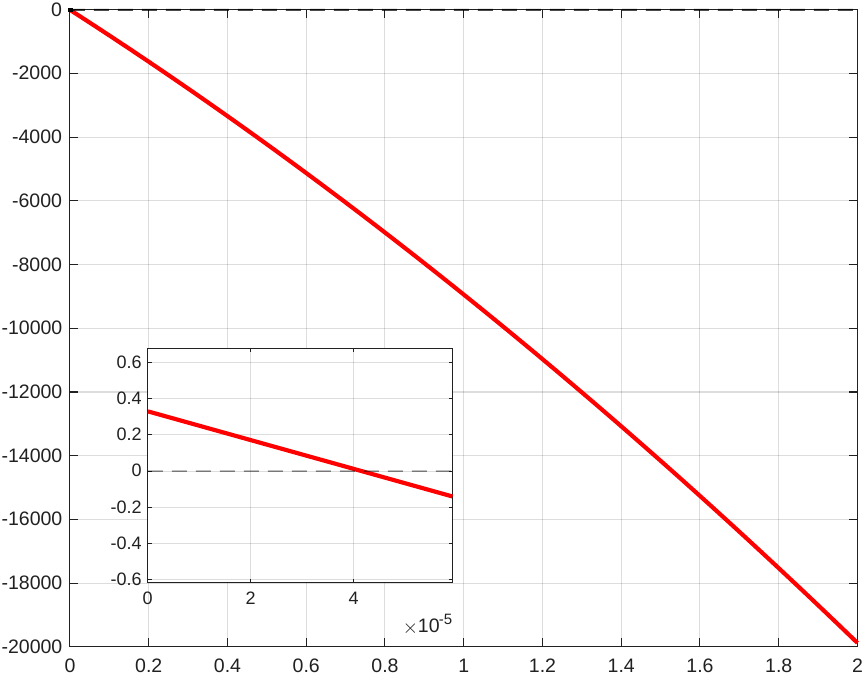}

\put(51.5,-6){\scalebox{1}{$\sigma$}}

\put(-6,40){\scalebox{1}{$\mu^\star$}}
\end{overpic}
  \end{minipage}\hfill
  \begin{minipage}[t]{0.30\textwidth}
    \centering
    \begin{overpic}[width=\linewidth]{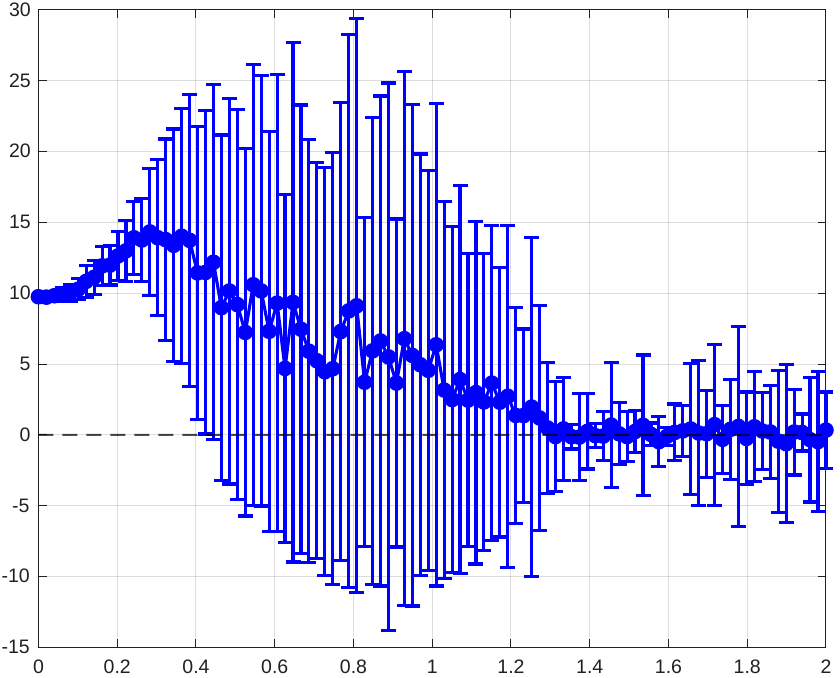}

\put(49.5,-6){\scalebox{1}{$\sigma$}}

\put(-8,40){\scalebox{1}{$\mu_\e$}}
\end{overpic}
  \end{minipage}\hfill
  \begin{minipage}[t]{0.30\textwidth}
    \centering
    \begin{overpic}[width=\linewidth]{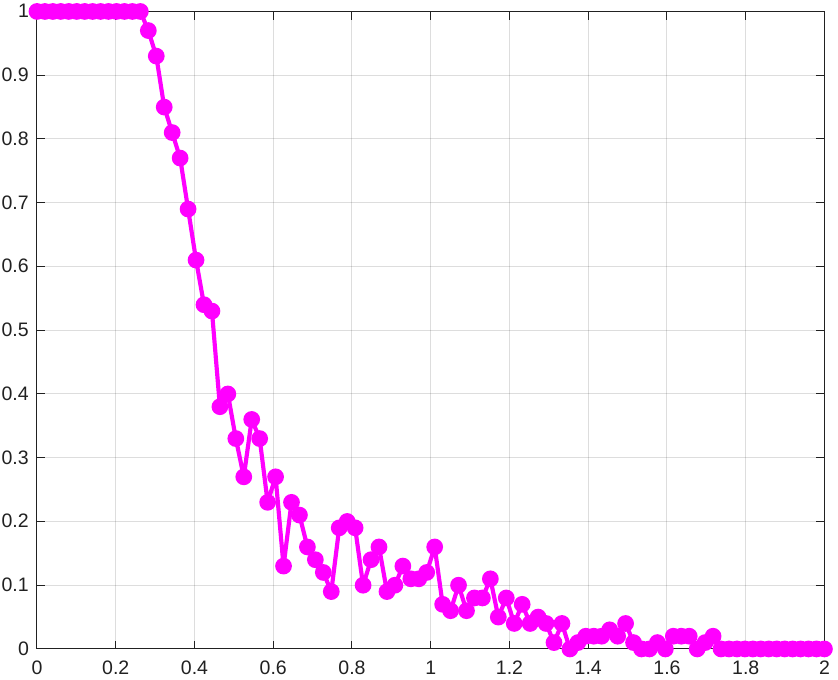}

\put(49.5,-6){\scalebox{1}{$\sigma$}}

\put(-8,40){\scalebox{1}{$\mathbb{P}$}}
\end{overpic}
  \end{minipage}
  \caption{\textcolor{black}{Plot of $\mu^\star,\mu_\e$ and $\mathbb{P}$ over $\sigma$ of initial data uniformly sampled from $[0,\pi]$}}
  \label{fig_8}
\end{figure}
\begin{figure}[H]
  \centering
  \begin{minipage}[t]{0.30\textwidth}
    \centering
    \begin{overpic}[width=\linewidth]{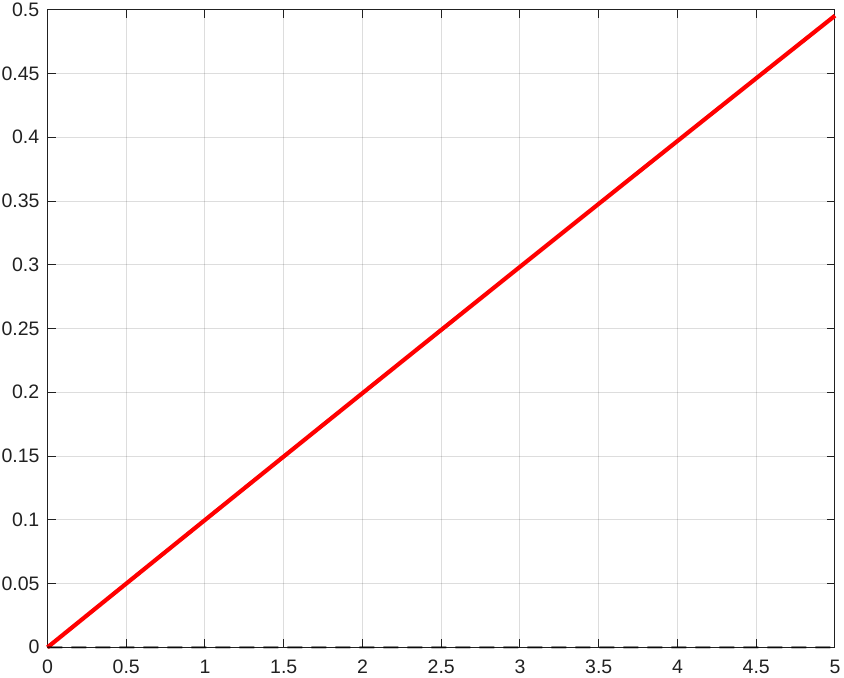}

\put(48.5,-7){\scalebox{1}{$K$}}

\put(-8,40){\scalebox{1}{$\mu^\star$}}
\end{overpic}
  \end{minipage}\hfill
  \begin{minipage}[t]{0.30\textwidth}
    \centering
    \begin{overpic}[width=\linewidth]{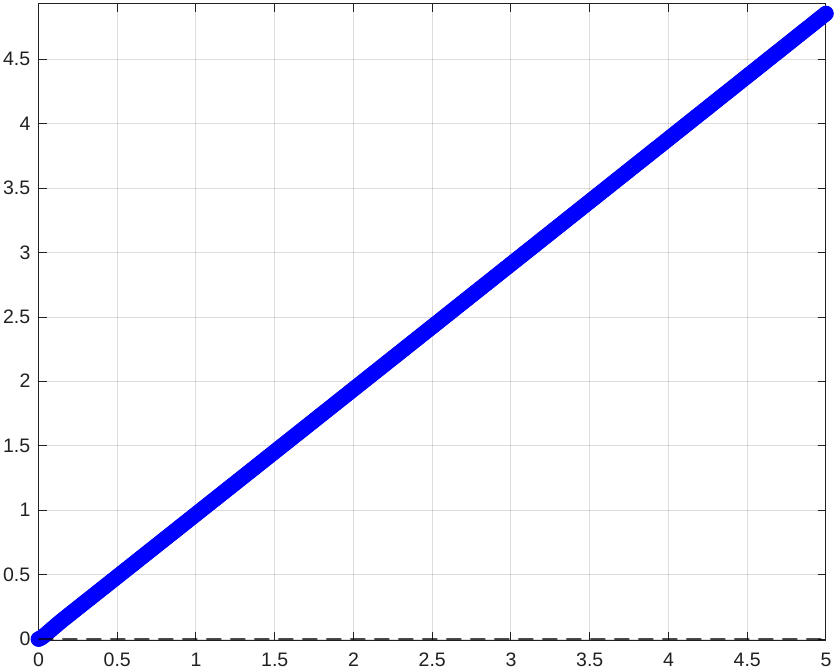}

\put(48,-7){\scalebox{1}{$K$}}

\put(-9,40){\scalebox{1}{$\mu_\e$}}
\end{overpic}
  \end{minipage}\hfill
  \begin{minipage}[t]{0.30\textwidth}
    \centering
    \begin{overpic}[width=\linewidth]{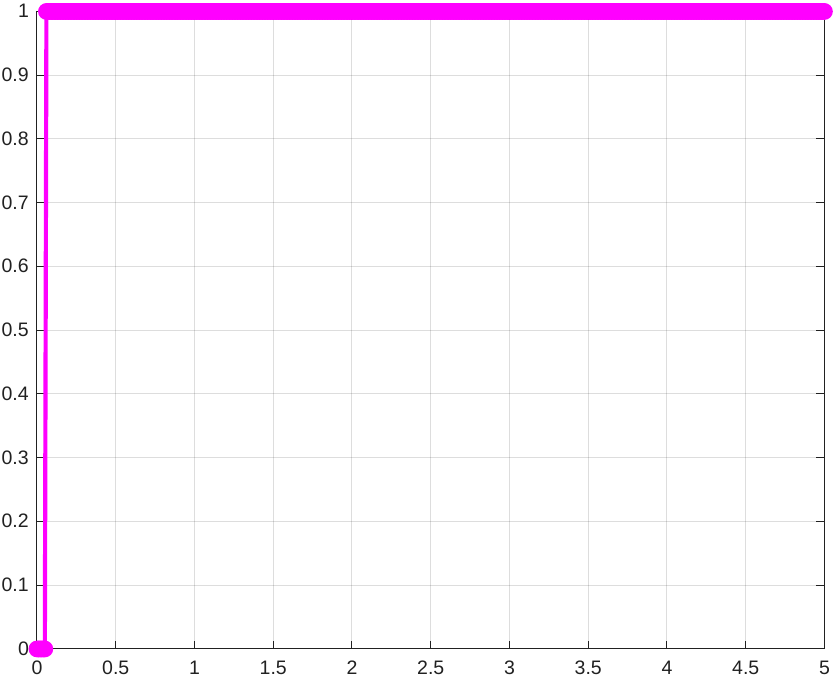}

\put(48,-7){\scalebox{1}{$K$}}

\put(-8,40){\scalebox{1}{$\mathbb{P}$}}
\end{overpic}
  \end{minipage}
  \caption{\textcolor{black}{Plot of $\mu^\star,\mu_\e$ and $\mathbb{P}$ over $K$ of initial data uniformly sampled from $[0,\pi]$}}
  \label{fig_9}
\end{figure}
\textcolor{black}{The numerical experiments indicate that the theoretical convergence rate $\mu^\star$ is not sharp. Indeed, the bound provided by Theorem \ref{main_th} predicts synchronization only for relatively small noise intensities, namely $\sigma$ of order $10^{-5}$--$10^{-3}$, and yields a convergence rate $\mu^\star$ that decreases monotonically as $\sigma$ increases.
In contrast, the simulations reveal a significantly more robust behaviour. The system is able to synchronize with high probability even for noise intensities as large as $\sigma=10^{-1}$, well beyond the theoretical regime covered by our results. Moreover, the empirical convergence rate $\mu_\e$ appears to increase with the noise intensity up to a certain threshold, suggesting that moderate levels of noise may actually enhance the synchronization process. 
A similar discrepancy is observed when varying the coupling strength $K$. While the qualitative dependence of the convergence rate on $K$ is correctly captured by the theory, the value of $\mu^\star$ remains substantially smaller than the empirical rate $\mu_\e$, typically by approximately one order of magnitude.  } 

\textcolor{black}{
These observations reveal a substantial discrepancy between the current theoretical estimates and the behaviour observed in the simulations. Although the present analysis provides a first quantitative framework for studying synchronization in the rough Kuramoto model, the numerical evidence suggests that the theoretical bounds are not yet optimal. Bridging this gap  requires the development of more refined analytical tools and a deeper understanding of the role played by rough noise in the synchronization mechanism. Finally, we mention that the synchronization speed naturally depends on $H$. In particular, numerical simulations yield that for $H<1/2$ the synchronization speed is faster than for $H\geq 1/2$. This fact is not surprising, since rougher noise determines the solution to exit quickly from contracting regions leading to a faster speed of convergence towards equilibrium. On the other hand, smoother noise makes the trajectories spend more time in contracting regions slowing down the speed of convergence. For example for SDEs driven by additive fractional Brownian motion, it is known that \cite[Theorem 1.2]{Hairer} for every initial datum, the corresponding solution converges to a stationary solution in the total variation norm and the speed of convergence is of order $t^{-\delta}$ where $\delta<\max_{\alpha<H} \alpha (1-2\alpha)$ and $H<1/2$. We leave these questions for the rough Kuramoto model for future investigation.}\\

We conclude by observing that our analysis has been carried out in the setting of rough noise, without relying on a specific probabilistic structure. In particular, our framework applies to random rough paths and therefore includes Brownian motion equipped with either its Itô or Stratonovich lift. While our results hold at this level of generality, the case of standard Brownian motion is expected to exhibit a richer interplay between the deterministic drift and the stochastic forcing. Indeed, when working with the Itô lift, one can exploit additional probabilistic tools such as the Itô calculus, martingale methods, and the Markov property, in order to obtain a more precise quantitative understanding of this interaction.
\subsection*{Acknowledgements}
 The authors have been supported by the  Deutsche Forschungsgemeinschaft (DFG, German Research Foundation) - Project ID 543163250.
\end{section}

\bibliographystyle{abbrv}
\bibliography{zReference}
\appendix

\begin{section}{A brief introduction to rough paths}\label{App_A}
Let us briefly present the concept of rough paths in the simplest form, following Friz \& Hairer \cite{friz2014course} and Lyons \cite{lyons1998differential}. For any finite dimensional vector space $X$, denote by $C([a, b], X)$ the space of all continuous paths $y:[a, b] \rightarrow X$ equipped with the sup norm $\|\cdot\|_{\infty,[a, b]}$ given by $\|y\|_{\infty,[a, b]}= \sup _{t \in[a, b]}\left\|y_t\right\|$, where $\|\cdot\|$ is the norm in $X$. We write $y_{s, t}:=y_t-y_s$. For $p \geq 1$, denote by $C^{p-\operatorname{var}}([a, b], X) \subset C([a, b], X)$ the space of all continuous paths $y:[a, b] \rightarrow X$ of finite $p$-variation
$$
\tb y\tb_{p-\operatorname{var},[a, b]}:=\left(\sup _{\Pi([a, b])} \sum_{i=1}^n\left\|y_{t_i, t_{i+1}}\right\|^p\right)^{1 / p}<\infty,
$$
where the supremum is taken over the whole class of finite partitions of $[a, b]$. Also for each $0< \alpha<1$, we denote by $C^\alpha([a, b], X)$ the space of Hölder continuous functions with exponent $\alpha$ on $[a, b]$ equipped with the norm
$$
\|y\|_{\alpha,[a, b]}:=\left\|y_a\right\|+\tb y\tb_{\alpha,[a, b]}, \quad \text { where } \quad\tb y\tb_{\alpha,[a, b]}:=\sup _{s, t \in[a, b], s<t} \frac{\left\|y_{s, t}\right\|}{(t-s)^\alpha}<\infty .
$$
For $\alpha \in\left(\frac{1}{3}, \frac{1}{2}\right)$, a couple $\mathbf{W}=(W, \mathbb{W}) \in \mathbb{R}^m \oplus\left(\mathbb{R}^m \otimes \mathbb{R}^m\right)$, where $W \in C^\alpha\left([a, b], \mathbb{R}^m\right)$ and
$$
\mathbb{W} \in C^{2 \alpha}\left([a, b]^2, \mathbb{R}^m \otimes \mathbb{R}^m\right):=\left\{\mathbb{W} \in C\left([a, b]^2, \mathbb{R}^m \otimes \mathbb{R}^m\right): \sup _{s, t \in[a, b], s<t} \frac{\left\|\mathbb{W}_{s, t}\right\|}{|t-s|^{2 \alpha}}<\infty\right\},
$$
is called a \textit{rough path} if it satisfies Chen's relation
$$
\mathbb{W}_{s, t}-\mathbb{W}_{s, u}-\mathbb{W}_{u, t}=W_{s, u} \otimes W_{u, t}, \quad \forall a \leq s \leq u \leq t \leq b .
$$
We introduce the rough path semi-norm
$$
\tb\mathbf{W}\tb_{\alpha,[a, b]}:=\tb W\tb_{\alpha,[a, b]}+\tb\mathbb{W}\tb_{2 \alpha,[a, b]^2}^{\frac{1}{2}}, \quad \text { where } \quad\tb\mathbb{W}\tb_{2 \alpha,[a, b]^2}:=\sup _{s, t \in[a, b] ; s<t} \frac{\left\|\mathbb{W}_{s, t}\right\|}{|t-s|^{2 \alpha}}<\infty
$$
A common choice is to fix parameters $\frac{1}{3}<\alpha<\frac{1}{2}$ and $p\geq\frac{1}{\alpha}$ so that $C^\alpha([a, b], X) \subset C^{p-\operatorname{var}}([a, b], X)$. We also set $q=\frac{p}{2}$ and consider the $p$-var semi-norm
$$\tb\mathbf{W}\tb_{p-\operatorname{var},[a, b]}:=\left(\tb W\tb_{p-\operatorname{var},[a, b]}^p+\tb\mathbb{W}\tb_{q-\operatorname{var},[a, b]^2}^q\right)^{\frac{1}{p}},$$
$$\tb\mathbb{W}\tb_{q-\operatorname{var},[a, b]^2}:=\left(\sup _{\Pi([a, b])} \sum_{i=1}^n\left\|\mathbb{W}_{t_i, t_{i+1}}\right\|^q\right)^{1 / q},
$$
where the supremum is taken over the whole class of finite partitions $\Pi([a, b])$ of $[a, b]$.
In the rough path setting, denote by $T_1^2\left(\mathbb{R}^m\right)=1 \oplus \mathbb{R}^m \oplus\left(\mathbb{R}^m \otimes \mathbb{R}^m\right)$ the set with the group product
$$
\left(1, g^1, g^2\right) \bullet\left(1, h^1, h^2\right)=\left(1, g^1+h^1, g^1 \otimes h^1+g^2+h^2\right)
$$
for all $\mathbf{g}=\left(1, g^1, g^2\right), \mathbf{h}=\left(1, h^1, h^2\right) \in T_1^2\left(\mathbb{R}^m\right)$. Given $\alpha \in\left(\frac{1}{3}, \frac{1}{2}\right)$, denote by $\mathscr{C}^{0, \alpha}\left(I, T_1^2\left(\mathbb{R}^m\right)\right)$ the closure of $\mathscr{C}^{\infty}\left(I, T_1^2\left(\mathbb{R}^m\right)\right)$ in the Hölder space $\mathscr{C}^\alpha\left(I, T_1^2\left(\mathbb{R}^m\right)\right)$, and by $\mathscr{C}_0^{0, \alpha}\left(\mathbb{R}, T_1^2\left(\mathbb{R}^m\right)\right)$ the space of all paths $\left.\mathbf{g}: \mathbb{R} \rightarrow T_1^2\left(\mathbb{R}^m\right)\right)$ such that $\left.\mathbf{g}\right|_I \in \mathscr{C}^{0, \alpha}\left(I, T_1^2\left(\mathbb{R}^m\right)\right)$ for each compact interval $I \subset \mathbb{R}$ containing 0. Assign $\Omega:=\mathscr{C}_0^{0, \alpha}\left(\mathbb{R}, T_1^2\left(\mathbb{R}^m\right)\right)$ and equip it with the Borel $\sigma$-algebra $\mathcal{F}$. Denote by $\Xi$ the \textit{Wiener-type shift}
$$
\left(\Xi_t \omega\right)_\cdot=\omega_t^{-1} \bullet \omega_{t+\cdot}, \qquad\forall t \in \mathbb{R}, \omega \in \mathscr{C}_0^{0, \alpha}\left(\mathbb{R}, T_1^2\left(\mathbb{R}^m\right)\right)
$$
and define the so-called \textit{diagonal process} $\mathbf{X}: \mathbb{R} \times \Omega \rightarrow T_1^2\left(\mathbb{R}^m\right), \mathbf{X}_t(\omega)=\omega_t$ for all $t \in \mathbb{R}, \omega \in \Omega$. Under assumption $\left(\mathbf{H}_W\right)$, it can be proved that there exists a probability measure $\mathbb{P}$ which is $\theta$ invariant \cite[Theorem 5]{bailleul2017random}. Thus $(\Omega, \mathcal{F}, \mathbb{P})$ is a probability space equipped with the continuous (thus measurable) metric dynamical system $\theta: \mathbb{R} \times \Omega \rightarrow \Omega$. In particular, the Wiener shift (5.5) implies that
$$
\left\tb\mathbf{W}\left(\Xi_h \omega\right)\right\tb_{p-\operatorname{var},[s, t]}=\tb\mathbf{W}(\omega)\tb_{p-\operatorname{var},[s+h, t+h]}
$$
Moreover, $\Xi$ is ergodic if $W=B^H$ is a fractional Brownian motion, see e.g.~\cite[Theorem 20]{GS}.
\end{section}

\begin{section}{Controlled Rough Paths}\label{App_gub}
The theory of controlled rough paths gives a meaning to (6) even if the paths of the driving process are
not smooth, but only $\alpha$-Hölder regular for $\alpha\in\left(1/3,1/2\right]$
\begin{definition}
    A path $Y \in C^{\alpha}([0,1] ; X)$ is controlled by $W \in C^{\alpha}([0,1] ; \R^m)$ if there exists $Y^{\prime} \in C^{\alpha}([0,1] ; \mathcal{L}(\R^m, X))$ such that
\begin{equation}
Y_{t}=Y_{s}+Y_{s}^{\prime} W_{s, t}+R_{s, t}^{Y} 
\end{equation}
where the remainder $R^{Y}$ has $2 \alpha$-Hölder regularity. The space of controlled rough paths $\left(Y, Y^{\prime}\right)$ on the time-interval $[0,1]$ is denoted by $D_{W}^{2 \alpha}([0,1] ; X)=: D_{W}^{2 \alpha}$. This space is endowed with the semi-norm
\begin{equation}
\tb Y, Y^{\prime}\tb_{D_{W}^{2 \alpha}}\coloneqq\tb Y^{\prime}\tb_{\alpha}+\tb R^{Y}\tb_{2 \alpha} 
\end{equation}
\end{definition}
A norm on $D_{W}^{2 \alpha}$ is given by $\left|Y_{0}\right|+\left|Y_{0}^{\prime}\right|+\tb Y^{\prime}\tb_{\alpha}+\tb R^{Y}\tb_{2 \alpha}$ or $\|Y\|_{\infty}+\left\|Y^{\prime}\right\|_{\infty}+\tb Y^{\prime}\tb_{\alpha}+\tb R^{Y}\tb_{2 \alpha}$, where both norms are equivalent, see \cite[Rem. 4.10]{brault2019solving}. The norm on $D_{W}^{2 \alpha}([0,1] ; X)$ is denoted by $\tb\cdot\tb_{\mathbf{D}_{W}^{2 \alpha}([0,1] ; X)}$ (short $\tb\cdot\tb_{\mathbf{D}_{W}^{2 \alpha}}$ ). Obviously, if $Y_{0}=0$ and $Y_{0}^{\prime}=0$, then the previous one is a norm on $D_{W}^{2 \alpha}$. The term $Y^{\prime}$ is referred to as the Gubinelli derivative of $Y$, see \cite{gubinelli2004controlling} and \cite[Ch. 4]{friz2014course}. For smooth paths $Y$ and $W$, the choice of $Y^{\prime}$ is not unique. However, one can show that for rough inputs $W, Y^{\prime}$ is uniquely determined by $Y$, see \cite[Rem. 4.7 and Sec. 6.2]{friz2014course}.
\begin{theorem}\hspace{-0.15cm}{(\cite[Prop. 1]{gubinelli2004controlling})}\label{gub_int}
Let $\left(Y, Y^{\prime}\right) \in D_{W}^{2 \alpha}\left([0,1] ; \mathbb{R}^{N \times m}\right)$ and $\boldsymbol{W}=(W, \mathbb{W})$ be an $\mathbb{R}^{m}$-valued $\alpha$-Hölder rough path for some $\alpha \in\left(\frac{1}{3}, \frac{1}{2}\right]$. Furthermore, $\mathcal{P}$ stands for a partition of $[0,1]$. Then, the integral of $Y$ against $\boldsymbol{W}$ defined as
\begin{equation}
\int_{s}^{t} Y_{r} \mathrm{~d} \boldsymbol{W}_{r}:=\lim _{|\mathcal{P}| \rightarrow 0} \sum_{[u, v] \in \mathcal{P}}\left(Y_{u} W_{u, v}+Y_{u}^{\prime} \mathbb{W}_{u, v}\right) 
\end{equation}
exists for every pair $s, t \in[0,1]$. Moreover, the estimate
\begin{equation}
\left|\int_{s}^{t} Y_{r} \mathrm{~d} \boldsymbol{W}_{r}-Y_{s} W_{s, t}-Y_{s}^{\prime} \mathbb{W}_{s, t}\right| \leq C\left(\tb W\tb_{\alpha}\tb R^{Y}\tb_{2 \alpha}+\tb\mathbb{W}\tb_{2 \alpha}\tb Y^{\prime}\tb_{\alpha}\right)|t-s|^{3 \alpha}
\end{equation}
holds true for all $s, t \in[0,1]$. The map from $\mathcal{D}_{W}^{2 \alpha}\left([0,1] ; \mathbb{R}^{N \times m}\right)$ to $D_{W}^{2 \alpha}\left([0,1] ; \mathbb{R}^{N}\right)$ given by
$$
\left(Y, Y^{\prime}\right) \mapsto\left(P, P^{\prime}\right):=\left(\int_{0}^{\cdot} Y_{r} \mathrm{~d} \boldsymbol{W}_{r}, Y .\right)
$$
is linear and continuous. Furthermore, the estimate
\begin{equation}
\tb P, P^{\prime}\tb_{D_{W}^{2 \alpha}} \leq\tb Y\tb_{\alpha}+\left\|Y^{\prime}\right\|_{\infty}\tb\mathbb{W}\tb_{2 \alpha}+C\left(\tb W\tb_{\alpha}\tb R^{Y}\tb_{2 \alpha}+\tb\mathbb{W}\tb_{2 \alpha}\tb Y^{\prime}\tb_{\alpha}\right)
\end{equation}
is valid.
\end{theorem}
We would like to end this Appendix stating few estimates that hold true for a general RDE like \eqref{RDE}. In particular, the solution of the RDE \eqref{RDE} is given by $(y,y')=(y,g(y))$ and satisfies the following estimates (see \cite[Theorem 8.3]{friz2014course} for a similar statement for the H\"older norm instead of the $p$-variation one)
\begin{equation}\label{est_g}
    \tb(g(y))'\tb_{p-\text{var},[s,t]}\leq 2C_g^2\tb y\tb_{p-\text{var},[s,t]}
\end{equation}
\begin{equation}\label{est_remainder}
    \tb R^{g(y)}\tb_{q-\text{var},[s,t]^2}\leq C_g\tb R^y\tb_{q-\text{var},[s,t]^2}+\frac{1}{2}C^2_g\tb x\tb_{p-\text{var},[s,t]}\tb y\tb_{p-\text{var},[s,t]}.
\end{equation}
         In terms of the Greedy times, recall~\cite[Section 4]{CLL} one can show that 
            \begin{align}\label{sol_bound}
        \begin{split}
            \tb y,R^y\tb_{p-\text{var},[s,t]}\leq&\left[\|y_s\|+\left(\frac{\|f(0)\|}{C_f}+\frac{1}{C_p}\textcolor{black}{\overline{\nN}}\left(\mathfrak{C},\textbf{x},[s,t]\right)\right)\right]\times\\
            &\hspace{4cm}\times e^{4C_f(t-s)}\textcolor{black}{\overline{\nN}}\left(\mathfrak{C},\textbf{x},[s,t]\right)^{\frac{p-1}{p}}-\|y_s\|
       \end{split}
       \end{align}
        where the left hand side is defined as
        $$\tb y,R^y\tb_{p-\text{var},[s,t]}\coloneqq\tb y\tb_{p-\text{var},[s,t]}+\tb R^y\tb_{q-\text{var},[s,t]^2}$$
        and $\mathfrak{C}$ a constant depending on the construction of the greedy times. Here we have that $\mathfrak{C}=\frac{1}{16C_{p}}$.
\end{section}

\begin{section}{Complementary proofs}\label{App_B}
\begin{proof}{(Lemma \ref{lemma1})}
    We can alternatively prove the statement by showing that 
    $$\frac{\textnormal{d}}{\textnormal{d}t}\sum_{i=1}^N z_i(t)=0$$
    since by definition have that
    $$\sum_{i=1}^N z_i(\tau_k)=\sum_{i=1}^N y_i(\tau_k)=0.$$
    In the following we will denote $\underline{1}=\textcolor{black}{\{1\}^N}$ \textcolor{black}{the $N$-dimensional vector whose entries are all equal to }$\textcolor{black}{1}$, \textcolor{black}{so that we want to prove that} $\dot{\underline{1}z(t)}=0$. One can prove that the Jacobian matrix $\frac{\partial \phi}{\partial z}$ satisfies
    $$\begin{cases}
    \displaystyle
        \d\frac{\partial \phi_{t,\tau_k}(\mathbf{W},z)}{\partial z}=\frac{\partial \tilde{G}}{\partial z}\left(\phi_{t,\tau_k}(\mathbf{W},z)\right)\frac{\partial \phi_{t,\tau_k}(\mathbf{W},z)}{\partial z}\d \textbf{W}_t,\\
           \displaystyle
        \frac{\partial \phi_{\tau_k,\tau_k}(\mathbf{W},z)}{\partial z}=\textnormal{Id}.
    \end{cases}$$
    This implies in particular that 
    \begin{align*}
        \d\left(\underline{1}\frac{\partial \phi_{t,\tau_k}(\mathbf{W},z)}{\partial z}\right)=\underline{1}\frac{\partial \Tilde{G}}{\partial z}&\left(\phi_{t,\tau_k}(\mathbf{W},z)\right)\frac{\partial \phi_{t,\tau_k}(\mathbf{W},z)}{\partial z}\d \textbf{W}_t=\\&=\sum_{j=1}^m\left(\begin{array}{c}
       \partial_{z_1}\sum_{i=1}^N \Tilde{G}_{ij} \\
       \vdots \\
       \partial_{z_N}\sum_{i=1}^N \Tilde{G}_{ij}
        \end{array}\right)\left(\phi_{t,\tau_k}(\mathbf{W},z)\right)\frac{\partial \phi_{t,\tau_k}(\mathbf{W},z)}{\partial z}\d \textbf{W}_t^j=0
    \end{align*}
    where we have used again that $$\sum_{i=1}^N \Tilde{G}_{ij}=0$$ for every $j$. So we obtain that
    $$\begin{cases}
    \displaystyle
        \d\left(\underline{1}\frac{\partial \phi_{t,\tau_k}(\mathbf{W},z)}{\partial z}\right)=0,\\
           \displaystyle
        \underline{1}\frac{\partial \phi_{\tau_k,\tau_k}(\mathbf{W},z)}{\partial z}=\underline{1}.
    \end{cases}$$
which obviously implies that
$$\underline{1}\left[\frac{\partial \phi_{t,\tau_k}(\mathbf{W},z)}{\partial z}\right]^{-1}
=\underline{1}\qquad\forall t\in[\tau_k,\tau_{k+1}].$$
From this last inequality we obtain the statement, namely
$$\underline{1}\dot{z(t)}=\underline{1}\left[\frac{\partial \phi_{t,\tau_k}(\mathbf{W},z)}{\partial z}\right]^{-1}f(\phi_{t,\tau_k}(\mathbf{W},z_t))=\underline{1}f(\phi_{t,\tau_k}(\mathbf{W},z_t))=0$$
where in the last equality we have used that we are taking $f$ as the vector field that defines the Kuramoto model.
\end{proof}
\end{section}
\end{document}